\documentclass[letterpaper,11pt,reqno]{amsart}
\usepackage{amsmath,amssymb,amsthm,epsfig,fancyvrb,graphicx,verbatim}

\usepackage[dvips]{color}
\usepackage{caption}
\usepackage{pgf}

\input xy  
\xyoption{all}
\entrymodifiers={+!!<0pt,\fontdimen22\textfont2>}

\numberwithin{equation}{section}
\numberwithin{figure}{section}

\theoremstyle{plain}

\newtheorem{Thm}{Theorem}[section]

\newtheorem{Prop}[Thm]{Proposition}
\newtheorem{Cor}[Thm]{Corollary}

\newtheorem{Lem}[Thm]{Lemma}

\theoremstyle{definition}

\newtheorem{Def}[Thm]{Definition}

\theoremstyle{remark}

\def\cgraphics#1{\raisebox{-.5\height}{\includegraphics{#1}}}  

\def\Itemref#1{(\uppercase\expandafter{\romannumeral 0\ref{#1}})}  

\def\eqspace{\hspace{.6cm}}

\def\cf{\mathbf{1}}
\def\EV{\mathbf{E}}
\def\Pr{\mathbf{P}}

\def\rm#1{\textrm{#1}}
\def\ul#1{\underline{#1}}
\def\BF#1{\mathbf{#1}}

\def\to{\rightarrow}
\def\downto{\downarrow}

\def\Qp{\phi} 
\def\Vp{\psi} 
\def\a{\alpha}
\def\b{\beta}
\def\g{\gamma}
\def\cP{\mathcal{P}}
\def\cT{\mathcal{T}}
\def\cS{\mathcal{S}}
\def\cL{\mathcal{L}}
\def\x{\ul{x}}
\def\bx{\mathbf{x}}

\def\bv{\mathbf{v}}
\def\tM{\widetilde{M}}

\def\e{\epsilon}
\def\d{\delta}

\def\Bbr{B^{\rm{br}}}
\def\Bex{B^{\rm{ex}}}

\def\Qbar{\overline{Q(S_n)}}	

\begin{document}

\title{The quantile transform of a simple walk}

\author[Assaf]{Sami Assaf}
\address{Department of Mathematics, University of Southern California,
  Los Angeles, CA 90089, USA}
\email{shassaf@usc.edu}

\author[Forman]{Noah Forman}
\address{Department of Mathematics, University of California Berkeley,
  Berkeley, CA 94720, USA}
\email{nforman@math.berkeley.edu}

\author[Pitman]{Jim Pitman}
\address{Department of Statistics, University of California Berkeley,
  Berkeley, CA 94720, USA}
\email{pitman@stat.berkeley.edu}

\date{\today}

\begin{abstract}
 We examine a new path transform on 1-dimensional simple random walks
 and Brownian motion, the \emph{quantile transform}. This
 transformation relates to identities in fluctuation theory due to
 Wendel, Port, Dassios and others, and to discrete and Brownian
 versions of Tanaka's formula. For an $n$-step random walk, the
 quantile transform reorders increments according to the value of the
 walk at the start of each increment. We describe the distribution of
 the quantile transform of a simple random walk of $n$ steps, using a
 bijection to characterize the number of pre-images of each possible
 transformed path. We deduce, both for simple random walks and for
 Brownian motion, that the quantile transform has the same
 distribution as Vervaat's transform. For Brownian motion, the
 quantile transforms of the embedded simple random walks converge to a
 time change of the local time profile. We characterize the
 distribution of the local time profile, giving rise to an identity
 that generalizes a variant of Jeulin's description of the local time
 profile of a Brownian bridge or excursion.
\end{abstract}

\maketitle

\setcounter{tocdepth}{1} 
\tableofcontents

\section{Introduction}

Given a simple walk with increments of $\pm1$, one observes
that the step immediately following the maximum value attained must be
a down step, and the step immediately following the minimum value
attained must be an up step. More generally, at a given value, the
subsequent step is more likely to be an up step the closer the value
is to the minimum and more likely to be a down step the closer the
value is to the maximum. To study this phenomenon more precisely, one
can form a two-line array with the steps of the walk and the value of
the walk, and then sort the array with respect to the values line and
consider the walk defined by the correspondingly re-ordered steps. It
is this transformation, which we term the \emph{quantile transform},
that we study here.

More precisely, for $w$ a walk, let $\Qp_w$ be the permutation of
$[1,n]$ such that, for $i < j$, either $w(\Qp_w(i)-1) <
w(\Qp_w(j)-1)$, or $w(\Qp_w(i)-1) = w(\Qp_w(j)-1)$ and $\Qp_w(i) <
\Qp_w(j)$. The \emph{quantile path transform} sends $w$ to the walk
$Q(w)$ where
\[  Q(w)(j) = \sum_{i=1}^j x_{\Qp_w(i)}. \]
In this paper, we characterize the image of the quantile transform on
simple (Bernoulli) random walks, which we call \emph{quantile walks},
and we find the multiplicity with which each quantile walk arises. These
results follow from a bijection between walks and \emph{quantile pairs}
$(v,k)$ consisting of a \emph{quantile walk} $v$ and a nonnegative integer $k$
satisfying certain conditions depending on $v$.

We also find, by passing to a Brownian limit, that the quantile transform of certain Bernoulli walks converge to an expression involving Brownian local times. This leads to a novel description of local times of Brownian motion up to a fixed time.

It is not difficult to describe the image of the set of walks under
the quantile transform; they are nonnegative walks and first-passage
bridges. Our main work is to prove the multiplicity with which each
image walk arises; this is stated in our Quantile bijection theorem, Theorem \ref{thm:q_bijection}, and illustrated in Figure \ref{fig:Q_helper}. We establish the bijection by decomposing the quantile transform into three maps:
\begin{align}\label{eq:q_decomp}
 (Q(w),\Qp_{w}^{-1}(n)) = \g\circ\b\circ\a(w).
\end{align}
In the middle stages of our sequence of maps we obtain combinatorial
objects which we call \emph{marked (increment) arrays} and
\emph{partitioned walks}.
\begin{align}\label{eq:q_decomp2}
 \rm{walk} \stackrel{\a}{\longmapsto} \rm{marked array}
 \stackrel{\b}{\longmapsto} \rm{partitioned walk}
 \stackrel{\g}{\longmapsto} \rm{walk-index pair}.
\end{align}
The three maps $\a$, $\b$, and $\g$ are discussed in sections
\ref{sec:incr_array}, \ref{sec:saw}, and \ref{sec:bijection_pf}
respectively.

In section \ref{sec:gen_walk} we prove an image-but-no-multiplicities
version of the Quantile bijection theorem for a more general class of
discrete-time processes.

In section \ref{sec:quantile_count} we show that the total number of
quantile pairs $(v,k)$ with $v$ having length $n$ is equal to the
number of walks of length $n$, i.e. $2^n$.

Section \ref{sec:incr_array} introduces \emph{increment arrays} and
defines the map $\a$. These arrays are a finite version of the
\emph{stack model} of random walk, which is the basis for
cycle-popping algorithms used to generate random spanning trees of
edge-weighted digraphs -- see Propp and Wilson\cite{PropWils98_1}.
Theorem \ref{thm:array_to_walk} asserts that $\a$ is injective and
characterizes its range; i.e.\ this theorem gives sufficient and
necessary conditions for a marked increment array to minimally
describe a walk.

In section \ref{sec:saw} we introduce \emph{partitioned walks} and the
map $\b$. This map is trivially a bijection, and Theorem
\ref{thm:saw_to_walk} describes the image of $\b\circ \a$. Equation
\eqref{eq:saw} is a discrete version of Tanaka's formula; this formula
has previously been studied in several papers, including
\cite{Kudzma82,CsorReve85_1,Szabados90,SzabSzek09}, and it plays a key
role both in this section and in the continuous setting.

In section \ref{sec:bijection_pf} we prove that $\g$ acts injectively
on the image of $\b\circ\a$, thereby completing our proof of Theorem
\ref{thm:q_bijection}.

Moving on from Theorem \ref{thm:q_bijection}, in section
\ref{sec:Vervaat} we demonstrate a surprising connection between the
quantile transform and a discrete version of the Vervaat transform,
which was discussed in definition \ref{def:Vervaat_cts}. Theorem
\ref{thm:Vervaat_wk} is the Vervaat analogue to Theorem
\ref{thm:q_bijection}. We find that quantile pairs and Vervaat pairs
coincide almost perfectly and that every walk has equally as many
preimages under the one transform as under the other.

In section \ref{sec:Jeulin}, we pass from simple random walks to a
Brownian limit in the manner of Knight\cite{Knight62,Knight63}. Our
path transformed walk converges strongly to a formula involving
Brownian local times. The bijection from the discrete setting results
in an identity, Theorem \ref{thm:gen_Jeulin}, describing local times of Brownian motion up to a fixed time, as a function of level. This identity generalizes a theorem of
Jeulin\cite{Jeulin85}.

Jeulin's theorem was applied by Biane and
Yor\cite{BianYor87} in their study of principal values around Brownian
local times. Aldous\cite{Aldous98}, too, made use of this identity to
study Brownian motion conditioned on its local time profile; and
Aldous, Miermont, and Pitman\cite{AldoMierPitm04}, while working in
the continuum random tree setting, discovered a version of Jeulin's
result for a more general class of L\'evy
processes. Leuridan\cite{Leuridan98} and Pitman\cite{Pitman99} have
given related descriptions of Brownian local times up to a fixed time,
as a function of level.

Related path transformations have been considered by Bertoin,
Chaumont, and Yor\cite{BertChauYor97} and later by
Chaumont\cite{Chaumont99} in connection with an identity of
fluctuation theory which had previously been studied by
Wendel\cite{Wendel60}, Port\cite{Port63}, and
Dassios\cite{Dassios95,Dassios96,Dassios05}. We conclude with a
discussion of these and other connections in section \ref{sec:background}.

\section{The quantile transform of a non-simple walk}
\label{sec:gen_walk}

It is relatively easy to describe the image of the quantile transform;
the difficulty lies in enumerating the preimages of a given image
walk. In this section we do the easy work, offering in Theorem
\ref{thm:q_non_simple} a weak version of Theorem \ref{thm:q_bijection}
in the more general setting of \emph{non-simple walks}. We conclude the section with a statement of our full Quantile bijection theorem, Theorem \ref{thm:q_bijection}.

Throughout this document we use the notation $[a,b]$ to denote an
interval of integers. While most results in the discrete setting apply
only to walks with increments of $\pm 1$, our results for this section apply to walks in general.

\begin{Def}
  For $n\geq 0$ a \emph{walk of length $n$} is a function $w: [0,n]\to
  \mathbb R$ with $w(0) = 0$.  We may view such a walk $w$ in terms of
  its increments, $x_i = w(i)-w(i-1)$, so that $w(j) = \sum_{i=1}^j
  x_i$.

  A walk of length $n$ is \emph{simple} if $w(i)-w(i-1) = \pm1$ for
  each $i\in [1,n]$. In particular, a simple walk is a function $w:
  [0,n]\to \mathbb Z$
\end{Def}

In subsequent sections of the document, for the sake of brevity we will say ``\emph{walk}'' to refer only to simple walks. 

\begin{Def}\label{def:q_perm}
  The \emph{quantile permutation} corresponding to a walk $w$ of
  length $n$, denoted $\Qp_w$, is the unique permutation of $[1,n]$
  with the property that
  $$(w(\Qp_w(1)-1),\Qp_w(1)-1);\ (w(\Qp_w(2)-1),\Qp_w(2)-1);\ \cdots;\ (w(\Qp_w(n)-1),\Qp_w(n)-1)$$
  is the increasing lexicographic reordering of the sequence
  $$(w(0),0);\ (w(1),1);\ \cdots;\ (w(n-1),n-1).$$

  The \emph{quantile path transform} sends $w$ to the walk $Q(w)$
  given by
  \begin{align}\label{eq:q_def}
    Q(w)(j) = \sum_{i=1}^j x_{\Qp_w(i)}\rm{\ for\ }j\in [1,n].
  \end{align}
\end{Def}

Note that the quantile permutation does not depend on the final
increment $x_n$ of $w$. A variant that does account for this final
increment was previously considered by Wendel\cite{Wendel60} and
Port\cite{Port63}, among others; this is discussed further in section
\ref{sec:background}.

We show an example of a simple walk and its quantile transform in Figure
\ref{fig:Q_transform}; for each $j$ the $j^{\rm{th}}$ increment of $w$
is labeled with $\Qp_w^{-1}(j)$. Observe that for a walk $w$ of length
$n$, we have $Q(w)(n) = w(n)$. As $j$ increases, the process
$Q(w)(j)$ incorporates increments which arise at higher values within
$w$. Consider example in Figure \ref{fig:Q_transform}. The first two
increments of $Q(w)$ correspond to the increments in $w$ which
originate at the value $-2$, the first five increments of $Q(w)$
correspond to those which originate at or below the value $-1$, and so
on.

\begin{figure}[htb]\centering
 \input{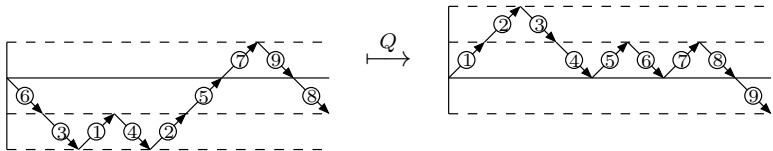}
 \caption{A walk and its quantile transform.}\label{fig:Q_transform}
\end{figure}

In discussing the proof and consequences of Theorem
\ref{thm:q_bijection} it is helpful to refer to several special
classes of walks.

\begin{Def}\label{def:special_walk}
  We have the following special classes of (simple) walks:
  \begin{itemize}
  \item A \emph{bridge to level $b$} is walk $w$ of length $n$ with
    $w(n) = b$; when $b = 0$, $w$ is simply a \emph{bridge}.
  \item A \emph{non-negative walk} is a walk of finite length which is non-negative at all times.
  \item A \emph{first-passage bridge} of length $n$ is a walk $w$
    which does not reach $w(n)$ prior to time $n$.
  \item A \emph{Dyck path} is a non-negative bridge (to level $0$).
  \end{itemize}
\end{Def}

As illustrated in Figure \ref{fig:q_alt_def}, $Q(w)(j)$ is the sum of
increments of $w$ which come from below a certain level. The graph of
$w$ is shown on the left and that of $Q(w)$ is on the right. The
increments which contribute to $Q(w)(6)$ are shown in both graphs as
numbered, solid arrows, and those that do not contribute are shown as
dashed arrows. The time $j=6$ is marked off with a vertical dotted
line on the left. Increments with their left endpoints strictly below
this value do contribute to $Q(w)(6)$, increments which originate at
exactly this value may or may not contribute, and increments which
originate strictly above this value do not contribute.

\begin{figure}[htb]\centering
 \input{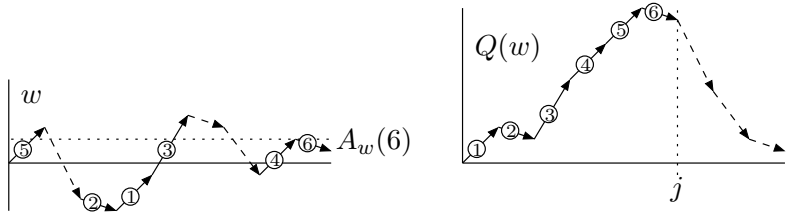}
 \caption{The value $Q(w)(6)$ is the sum of increments of $w$ which originate below $A_w(6)$, as well as some which originate exactly at $A_w(6)$.}\label{fig:q_alt_def}
\end{figure}

\begin{Def}
 Given a walk $w$, for $j\in [1,n]$ we define the \emph{quantile
   function of occupation measure}
 \begin{align*}
  A_w(j) &:= w(\Qp_w(j)-1).
 \end{align*}
\end{Def}

The quantile function of occupation measure may also be expressed
without reference to the quantile permutation by
\begin{align*}
 A_w(j) = \min\{a \in \mathbb{R}\ :\ \#\{i\in [0,n-1]\ :\ w(i) \leq
 a\}\geq j\}.
\end{align*}
On the left in Figure \ref{fig:q_alt_def}, the horizontal dotted line
indicates $A_w(6)$.

\begin{Thm}\label{thm:q_non_simple}
 For any walk $w$ of length $n$,
 \begin{align}
  \begin{array}{ll}
   Q(w)(j) \geq 0		& \rm{for}\ j \in [0,\ \Qp_w^{-1}(n))\rm{, and}\\
   Q(w)(j) > Q(w)(n)	& \rm{for}\ j \in [\Qp_w^{-1}(n),\ n).
  \end{array}
 \end{align}
 Consequently, $Q(w)$ is either a non-negative walk in the case where
 $w(n)\geq 0$ or a first-passage bridge to a negative value in the
 case where $w(n) < 0$.
\end{Thm}

\begin{proof}
 First we prove that for $j < \Qp_w^{-1}(n)$ we have $Q(w)(j) \geq
 0$. Afterwards, we prove that for $j \in [\Qp_w^{-1}(n),\ n)$ we
   have $Q(w)(j) > Q(w)(n)$.

 Fix $j < \Qp_w^{-1}(n)$ and let
 \begin{align*}
  I &:= \{i \in [1,n]\ :\ \rm{either\ }w(i-1) <
  A_w(j)\rm{\ or\ }w(i-1)=A_w(j)\rm{\ and\ }i\leq \Qp_w(j)\}.
 \end{align*}
 Thus
 \begin{align*}
  Q(w)(j) = \sum_{i\in I} x_i.
 \end{align*}

 We partition $I$ into maximal intervals of consecutive integers. For
 example, in Figure \ref{fig:gen_walk_blocks} with $j=6$ we have $I =
 \{1,2,4,5,8,9\}$, which comprises three intervals:
 $\{1,2\}$,\ $\{4,5\}$, and $\{8,9\}$. We label these intervals
 $I_1,\ I_2,$ and so on.

 \begin{figure}[htb]\centering
  \input{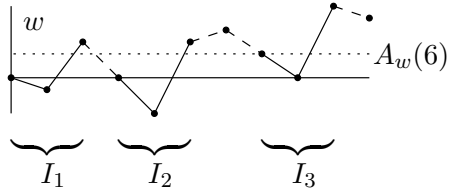}
  \caption{Three segments of the path of $w$ correspond to the three intervals in $I$.}\label{fig:gen_walk_blocks}
 \end{figure}

 These intervals correspond to segments of the path of $w$, shown in
 solid lines in the figure. Each such segment begins at or below
 $A_w(j)$ and each ends at or above $A_w(j)$. Here we rely on our
 assumption that $j < \Qp_w^{-1}(n)$ and thus $n\not\in I$: if one of
 our path segments included the final increment of $w$ then that
 segment might end below $A_w(j)$.

 Thus, for each $k$ we have
 $$\sum_{i\in I_k}x_i \geq 0,$$
 and so
 \begin{align*}
  Q(w)(j) &= \sum_{i\in I}x_i = \sum_k\sum_{i\in I_k}x_i \geq 0.
 \end{align*}

 Now fix $j \in [\Qp_w^{-1}(n),n)$, and we must show that $Q(w)(j) >
   Q(w)(n)$. Let $I^c$ denote $[1,n] - I$. Thus,
 \begin{align*}
  Q(w)(n) - Q(w)(j) = \sum_{i\in I^c}x_i.
 \end{align*}

 As with $I$ above, we partition $I^c$ into maximal intervals of
 consecutive numbers, $I^c_1,\ I^c_2,\ \cdots$. These intervals
 correspond to segments of the path of $w$. Each such segment begins
 at or above and ends at or below $A_w(j)$. As in the previous case,
 here we rely on our assumption that $j\geq \Qp_w^{-1}(n)$: if one of
 the $I^c_k$ included the final increment then the corresponding path
 segment might end above $A_w(j)$.

 Moreover if one of these segments begins exactly at $A_w(j)$ then it
 must end strictly below $A_w(j)$. In order for the segment
 corresponding to some block $[l,l+1,\cdots,m]$ of $I^c$ to begin
 exactly at $A_w(j)$ we would need: (1) $w(l-1) = A_w(j) =
 w(\Qp_w(j)-1)$ and (2) $l\in I^c$. Thus, by definition of $I$, we
 would have $l \geq \Qp_w^{-1}(j)$. And since $m+1 \in I$ and $m+1 >
 \Qp_w^{-1}(j)$, we would then have $w(m) < A_w(j)$, as claimed. We
 conclude that for each block $I^c_k$,
 $$\sum_{i\in I^c_k}x_i < 0.$$
 Consequently,
 \begin{align*}
  Q(w)(n) - Q(w)(j) = \sum_{i\in I^c}x_i = \sum_k\sum_{i\in I^c_k}x_i < 0,
 \end{align*}
 as desired.
\end{proof}

Theorem \ref{thm:q_non_simple} motivates the following definition.

\begin{Def}
  A \emph{quantile walk} is a simple walk that is either non-negative or a
  first-passage bridge to a negative value.

  A \emph{quantile pair} is a pair $(v,k)$ where $v$ is a quantile
  walk of length $n$ and $k$ is a nonnegative integer such that $v(j)
  \geq 0$ for $j\in [0,k)$ and $v(j) > v(n)$ for $j\in [k,n)$.
\end{Def}



The following is our main result in the discrete setting.

\begin{Thm}[Quantile bijection]\label{thm:q_bijection}
 The map $w \mapsto (Q(w), \Qp_w^{-1}(n))$ is a bijection between the set of simple walks of length $n$ and the set of quantile pairs $(v,k)$ with $v$ having length $n$.
\end{Thm}

This theorem is proved at the end of section \ref{sec:bijection_pf}. The next several sections build tools for that proof in the manner described in the introduction.

The index $\Qp_w^{-1}(n)$ serves as a \emph{helper variable} in the statement of the theorem, distinguishing between walks that have the same $Q$-image. This helper variable is the time at which the increment corresponding to the final increment of $w$ arises in $Q(w)$.

Figure \ref{fig:Q_helper} illustrates which indices $k$ may appear as
helper variables alongside a particular image walk $v$, depending on
the sign of $v(n)$. If $v(n) < 0$ then its helper $k$ may be any time
from 1 up to the hitting time of $-1$. If $v(n)\geq 0$ and $v$ ends in
a down-step then $k$ may be any time in the final excursion above the
value $v(n)$, including time $n$. In the special case where $v(n)\geq
0$ and $v$ ends with an up-step, $k$ can only equal $n$.

\begin{figure}[htb]\centering
 \input{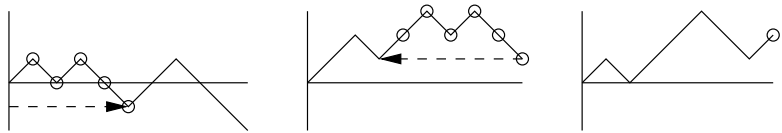}
 \caption{The allowed times for the helper variable (circled).\label{fig:Q_helper}}
\end{figure}

Throughout the remainder of the document we say ``\emph{walk}'' to refer to simple walks.

\section{Enumeration of quantile pairs}
\label{sec:quantile_count}

In this section we show that there are as many quantile pairs $(v,k)$ in which $v$ has $u$ up-steps and $d$-down steps as there are walks with $u$ up-steps and $d$ down-steps. We begin with notation.

Let $q(u,d)$ denote the number of quantile pairs $(v,k)$ for which $v$
has exactly $u$ up-steps and $d$ down-steps.  For $u\geq d$ let
$walk_+(u,d)$ denote the number of everywhere non-negative walks with
$u$ up-steps and $d$ down-steps.  For $u\not=d$ let $fpb(u,d)$ denote
the number of first-passage bridges with $u$ up-steps and $d$
down-steps.

The following two formulae are well known and can be found in
Feller\cite[p.\ 72-77]{Feller1}.
\begin{align}
 \rm{walk}_+(u,d) &= \binom{u+d}{u} - \binom{u+d}{u+1}\rm{ and}\label{eq:walk_+}\\
 \rm{fpb}(u,d) &= \binom{u+d-1}{u\wedge d} - \binom{u+d-1}{(u\wedge d)-1}.\label{eq:fpb}
\end{align}
A discussion of these and other formulae in this vein may also be found in \cite{EgecKing99}.

We call upon a version of the Cycle lemma.

\begin{Lem}[Cycle lemma, Dvoretzky and Motzkin, 1947\cite{DvorMotz47}]
 A uniformly random first-passage bridge to some level $-b$, with
 $b>0$, may be decomposed into $b$ consecutive, exchangeable random
 first-passage bridges to level $-1$. If we condition on the lengths
 of these first-passage bridges then they are independent and
 uniformly distributed in the sets of first-passage bridges to $-1$ of
 the appropriate lengths.
\end{Lem}

Versions of this lemma have been rediscovered many times. For more
discussion on this topic see \cite{DersZaks90} and
\cite[p.\ 172-3]{Pitman98} and references therein.

Finally, we require the following formula.
\begin{Lem}
 For any non-negative integers $u$ and $d$,
 \begin{align}
  q(u,d+1) &= q(d,u+1) - \binom{u+d}{u+1} + \binom{u+d}{u-1}.\label{eq:q_count_flip}
 \end{align}
\end{Lem}

\begin{proof}
 The formula is trivial in the case $u = d$. Moreover, it suffices to
 prove the formula in the case $u > d$, since the case $u < d$ follows
 by swapping variables.

 We define a bijective path transformation $T$ which transforms a
 non-negative walk ending in a down-step to a first-passage bridge
 down. This transformation offers a near duality between two classes of quantile pairs.

 Let $v$ be a non-negative walk that ends in a down step. We define a
 bijective path transformation $T$ which transforms such a walk into a
 first passage bridge down. In particular, $T$ transforms $v$ by the
 following three steps: (1) it removes the final down-step of $v$; (2)
 it reverses the sign and order of the remaining increments in $v$;
 and (3) it adds a final down-step to the resulting walk. This is
 illustrated in Figure \ref{fig:q_count_xform}.

 \begin{figure}\centering
  \input{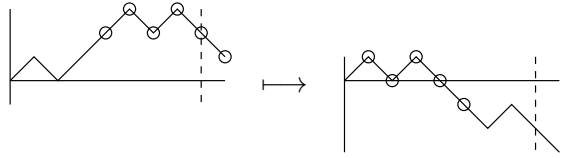}
  \caption{A path transform which almost preserves number of allowed helper values.}\label{fig:q_count_xform}
 \end{figure}

 Fix $u > d$. The transformation $T$ bijectively maps: (1)
 non-negative walks that end in down-steps and take $u$ up-steps and
 $d+1$ down-steps to (2) first-passage bridges that take $d$ up-steps
 and $u+1$ down-steps. This map has the additional property that $v$
 belongs to exactly one more quantile pair than $T(v)$ does:
 \begin{align}
  \#\{k : (v,k)\rm{ is quantile}\} = \#\{k  : (T(v),k)\rm{ is quantile}\} + 1.\label{eq:path_xform_q_count}
 \end{align}
 This gives the following identity for $u > d$:
 \begin{align}
  q(u,d+1) - \rm{walk}_+(u-1,d+1) = q(d,u+1) + \rm{fpb}(d,u+1).
 \end{align}
 The second term on the right corresponds to the ``$+ 1$'' from
 equation \eqref{eq:path_xform_q_count}. The second term on the left
 accounts for quantile pairs involving non-negative walks that end in
 up-steps. Subbing in the known counts \eqref{eq:walk_+} and
 \eqref{eq:fpb} gives the desired result.
\end{proof}

We now have all of the elements needed to prove our enumeration of quantile pairs.

\begin{Prop}\label{prop:quantile_lvl_count}
 For any non-negative integers $u$ and $d$,
 \begin{align}\label{eq:quantile_lvl_count}
  q(u,d) = \binom{u+d}{u}.
 \end{align}
\end{Prop}

\begin{proof}
 We  prove the result in the case $u<d$ and then use equation
 \eqref{eq:q_count_flip} to pass our result to the case where $u\geq
 d$.

 Suppose $u<d$. Let $(W_j,\ j\in[0,n])$ denote a uniform random first
 passage bridge conditioned to have $u$ up-steps and $d$ down-steps,
 where $u$ and $d$ are fixed. Let $T$ denote the first-arrival time of
 $W$ at $-1$ -- this is the random number of quantile pairs to which
 $W$ belongs. By the Cycle Lemma, $W$ may be decomposed into $d-u$
 exchangeable first passage bridges to $-1$. Thus,
 $$\EV(T) = \frac{u+d}{d-u}.$$
 So
 \begin{align*}
  q(u,d) &= \EV(T)\rm{fpb}(u,d)\\
	&= \frac{u+d}{d-u}\left(\binom{u+d-1}{d-1} - \binom{u+d-1}{u-1}\right)\\
	&= \frac{u+d}{d-u}\left(\binom{u+d}{d}\frac{d}{u+d} - \binom{u+d}{u}\frac{u}{u+d}\right) = \binom{u+d}{d},
 \end{align*}
 as desired.

 Now suppose $u \geq d$. By equation \eqref{eq:q_count_flip} and the
 previous case
 \begin{align*}
  q(u,d) &= \binom{u+d}{u+1} - \binom{u+d-1}{u+1} + \binom{u+d-1}{u-1}\\
	&= \binom{u+d-1}{u} + \binom{u+d-1}{u-1} = \binom{u+d}{u}.
 \end{align*}
\end{proof}

\section{Increment arrays}
\label{sec:incr_array}

The increment array corresponding to a walk is a collection of
sequences of $\pm 1$s, with each sequence listing the increments from
a particular \emph{level} of that walk. This is a finite version of
the stack model of a Markov process, discussed in Propp and
Wilson\cite[p.\ 205]{PropWils98_1} in connection with the cycle
popping algorithm for generating a random spanning tree of an
edge-weighted digraph. Whereas the stack model assumes an infinite
excess of instructions, we study increment arrays which minimally
describe walks of finite length. Theorem \ref{thm:array_to_walk}
characterizes these increment arrays.

In terms of the decomposition of $Q$ proposed in
equations \eqref{eq:q_decomp} and \eqref{eq:q_decomp2}, this section
defines and studies the map $\a$.

By virtue of their finiteness, increment arrays may be viewed as
discrete local time profiles with some additional
information. Discrete local times have been studied extensively; see,
for example, Knight\cite{Knight63} and R\'ev\'esz\cite{Revesz}. A more
complete list of references regarding asymptotics of discrete local
times is given in section \ref{sec:Jeulin}.

The quantile transform rearranges increments on the basis of their
left endpoints.

\begin{Def}
 Let $w$ be a walk of length $n$. For $1\leq j\leq n$ we define the
 \emph{level} of (the left end of) the $j^{\rm{th}}$ increment of $w$
 to be
 \begin{align*}
  w(j-1) - \min_{0\leq i<n} w(i).
 \end{align*}
 The $j^{\rm{th}}$ increment of $w$ is said to \emph{belong to}, or to
 \emph{leave}, that level. We name four important levels of a walk
 $w$, illustrated in Figure \ref{fig:level_labels}.
 \begin{itemize}
  \item The \emph{start level} is the level of the first increment, or
    $-\min_{i<n} w(i)$. We  typically denote this $\cS$, or
    $\cS_w$ in case of ambiguity.
  \item The \emph{terminal level} is $(w(n) - \min_{i<n} w(i))$. We
     typically denote this $\cT$ or $\cT_w$.
  \item The \emph{preterminal level} is the level of the final
    increment, or $(w(n-1) - \min_{i<n} w(i))$. We  typically
    denote this $\cP$ or $\cP_w$.
  \item The \emph{maximum level} is $\max_{j\in [0,n-1]} w(j)$. We
     typically denote this $\cL$ or $\cL_w$.
 \end{itemize}
\end{Def}

\begin{figure}[htb]\centering
 \input{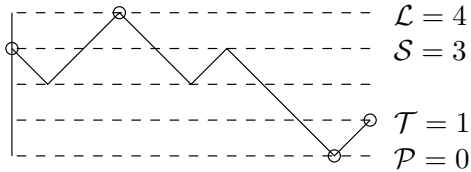}
 \caption{A walk with its distinguished levels labeled.}\label{fig:level_labels}
\end{figure}

Note that if $w$ is a first-passage bridge then no increments leave
its terminal level. In this case $\cT$ equals either $-1$ or $\cL +
1$. Because $\cT$ attains these exceptional values, the set of
first-passage bridges arise as a special case throughout this
document.

The start, preterminal, and terminal levels share the following relationship.
\begin{align}\label{eq:special_levels}
 \cS = \cT - w(n) = \cP - w(n-1).
\end{align}

The quantile transform of a walk $w$ is determined by the levels at
which the increments of $w$ occur and the orders in which they occur
at each level. We define \emph{increment arrays} to carry this
information.

\begin{Def}
 An \emph{increment array} is an indexed collection $\bx =
 (\x_i)_{i=0}^{\cL}$ of non-empty, finite sequences of $\pm 1$s. We
 call the $\x_i$s the \emph{rows} and $\cL$ the \emph{height} of the
 array. We say that an increment array $(\x_i)_{i=0}^{\cL}$
 \emph{corresponds to a walk $w$} with maximum level $\cL$ if, for
 every $i\in [0,\cL]$, the sequence of increments of $w$ at level $i$
 equals $\x_i$; i.e.
 $$\x^w_i = (w(s_1+1)-w(s_1)),\;\cdots,\; w(s_k+1) - w(s_k)),$$ where
 $s_1<\cdots<s_k$ is the sequence of times prior to $n$ at which $w$
 visits level $i$.
\end{Def}

An example of a walk and its corresponding increment array is given in
Figure \ref{fig:incr_array}. In that figure we've bolded the
increments from level 4.

\begin{figure}[htb]\centering
 \input{profile_array.pstex_t}
 \caption{A walk with the corresponding increment array and up- and down-crossing counts.\label{fig:incr_array}}
\end{figure}

\begin{Def}\label{def:crossing_counts}
 Given an increment array $\bx$, we define $u^{\bx}_i$ and $d^{\bx}_i$
 to be the number of `$1$'s and `$-1$'s, respectively, that appear in
 $\x_i$. Correspondingly, for a walk $w$ we define $u^w_i$ and $d^w_i$
 to be the numbers of up- and down-steps of $w$ from level $i$. We
 call the $u^\bx_i$s and $d^\bx_i$s (respectively $u^w_i$s and
 $d^w_i$s) the \emph{up-} and \emph{down-crossing counts} of $\bx$
 (resp.\ of $w$). We define the \emph{sum of $\bx$}, denoted
 $\sigma_{\bx}$, to be the sum of all increments in the array:
 \begin{align}
  \sigma_{\bx} := \sum_{i=0}^{\cL}\sum_{j\in \x_i} j = \sum_{i=0}^{\cL} u_i - d_i.
 \end{align}
\end{Def}

Clearly, if $\bx$ corresponds to a walk $w$ of length $n$ then $\sigma_{\bx} = w(n)$, and for each $i$
\begin{align*}
 u^{\bx}_i = u^w_i\rm{\;\; and\;\;} d^{\bx}_i = d^w_i.
\end{align*}

We now define the map $\a$, which was referred to in equations
\eqref{eq:q_decomp} and \eqref{eq:q_decomp2}. We need this map to be
injective, but we see in Theorem
\ref{thm:array_to_walk_multiplicity} that the map from a walk to its
corresponding increment array is not injective, so $\a(w)$ must pass
some additional information.

\begin{Def}\label{def:alpha}
 Given an increment array $\bx = (\x_i)_{i=0}^{\cL}$, we may
 arbitrarily specify one row $\x_{\cP}$ with $\cP\in [0,\cL]$ to be
 the \emph{preterminal row}. We call the pair $(\bx,\cP)$ a
 \emph{marked (increment) array}, since one row has been ``marked'' as
 the preterminal row. We say that the marked array \emph{corresponds
 to a walk $w$} if $w$ corresponds to $\bx$ and has preterminal
 level $\cP$.

 We define $\a$ to be the map which sends a walk $w$ to its
 corresponding marked array.
\end{Def}

Equation \eqref{eq:special_levels} may be restated in this setting. If
an array $\bx$ corresponds to a walk $w$ with preterminal level $\cP$
then the start and terminal levels of $w$ are specified by
\begin{align}\label{eq:special_rows}
 \cT = \cP - x_{\cP}^*\rm{,\;\; and\;\; }\cS = \cT - \sigma_{\bx},
\end{align}
where $x_{\cP}^*$ denotes the final increment in the row $\x_{\cP}$.

\begin{Def}
 For a marked array $(\bx,\cP)$ we define the indices $\cS$ and $\cT$
 via equation \eqref{eq:special_rows}. If $\cS$ falls within $[0,\cL]$
 then we call $\x_{\cS}$ the \emph{start row} of $\bx$; otherwise we
 say that the start row is empty. Likewise, if $\cT\in [0,\cL]$ then
 we call $\x_{\cT}$ the \emph{terminal row}, and if not then we say
 that the terminal row is empty.
\end{Def}

In Figure \ref{fig:reconst_alg} we state an algorithm to reconstitute
the walk corresponding to a valid marked array. This is the same
algorithm implied by the stack model of random walks, discussed in
\cite{PropWils98_1}. In light of this algorithm, a marked increment
array may be viewed as a set of instructions for building a walk: the
row $\x_i$ tells the walk which way to go on successive visits to
level $i$. Figure \ref{fig:reconst_eg} presents an example run of this
algorithm.

\begin{figure}[ptbh]\centering
 \begin{Verbatim}[numbers=left,numbersep=4pt,commandchars=\\\{\}]
Reconstitution(x[],P)
  ## Takes two arguments - incr array x[] and preterm lvl P
  ## Each x[i] is a queue w/ operation Pop(x[i]) which pops
  ##  x[i][0] off of x[i] and returns the popped value.

  L := length(x) - 1  ## set max level
  S := P + x[P+1][length(x[P+1])-1]  ## define S via \eqref{eq:special_rows}

  w[0] := 0, m := 0, i:= S

  While x[i] not empty:\label{line:reconst_loop}
    x := Pop(x[i])
    w[m+1] := w[m] + x
    i := i+x, m := m+1

  Return w
\end{Verbatim}
 \caption{A pseudocode algorithm to reconstitute a walk from a marked array.}\label{fig:reconst_alg}
\end{figure}

\begin{figure}[ptbh]
 \begin{tabular*}{\textwidth}[h]{rll}
	$\cP = 3$;	& $\x_0 = (1),\; \x_1 = (-1,1),\; \x_2 = (1),\; \x_3 = (-1)$\\[6pt]
  (1) $i = 1$;	& $\x_0 = (1),\; \x_1 = (\BF{-1},1),\; \x_2 = (1),\; \x_3 = (-1)$& \cgraphics{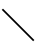}\\[6pt]
  (2) $i = 0$;	& $\x_0 = (\BF{1}),\; \x_1 = (1),\; \x_2 = (1),\; \x_3 = (-1)$	& \cgraphics{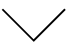}\\[6pt]
  (3) $i = 1$;	& $\x_1 = (\BF{1}),\; \x_2 = (1),\; \x_3 = (-1)$			& \cgraphics{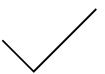}\\[6pt]
  (4) $i = 2$;	& $\x_2 = (\BF{1}),\; \x_3 = (-1)$						& \cgraphics{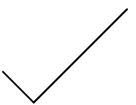}\\[6pt]
  (5) $i = 3$;	& $\x_3 = (\BF{-1})$										& \cgraphics{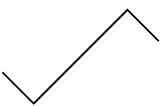}\\[18pt]
 \end{tabular*}
 \caption{Reconstitution algorithm (Fig.\ \ref{fig:reconst_alg}) run on a valid marked array (see Def.\ \ref{def:valid_array}). Input $(\bx,\cP)$ shown at top. Each row below corresponds to an iteration of the loop.
 }\label{fig:reconst_eg}
\end{figure}

We wish to characterize which marked arrays correspond to walks. This
is the main result of section \ref{sec:incr_array}.

\begin{Def}\label{def:valid_array}
  An increment array has the \emph{Bookends property} if for every $i
  \leq \min\{\cP,\cT\}$ the final entry in $\x_i$ is a $1$, and for
  each $i \geq \max\{\cP,\cT\}$ the final entry is a $-1$.

 A marked array has the \emph{The Crossings property} if for each
 $i\in [0,\cL+1]$
 \begin{align}
   u_{i-1} - d_i = \cf\{i\leq \cT\} - \cf\{i\leq \cS\},\label{eq:up_down_counts_array}
 \end{align}
 where we define $u_{-1} = d_{\cL+1} = 0$.

 A marked array with the Bookends and Crossings properties is called
 \emph{valid}. We call an increment array $\bx$ \emph{valid} if
 $(\bx,\cP)$ is valid for some $\cP$.
\end{Def}

\begin{Thm}\label{thm:array_to_walk}
 The map $\a$ is a bijection between the set of walks and the set of
 valid marked arrays.
\end{Thm}

The necessity of the Bookends property is clear. For each $i\not= \cP$
the last increment from level $i$ of a walk $w$ must go towards the
preterminal level. Likewise, for each $i\not= \cT$ the last increment
from level $i$ must go towards the terminal level. Note that because
there can be no index $i$ strictly between $\cP$ and $\cT$, these two
requirements are never in conflict.

Next we consider decomposing a walk around its visits to a level. We
use this idea first to prove the necessity of the Crossings property,
and then to prove the sufficiency of the conditions in Theorem
\ref{thm:array_to_walk}. This approach is motivated by excursion
theory and by the approach in Diaconis and Freedman\cite{DiacFree80},
which deals with related issues. In particular, whereas our Theorem
\ref{thm:array_to_walk} gives conditions for the existence of a path
corresponding to a given set of instructions (a marked array), Theorem
(7) in \cite{DiacFree80} gives conditions, based on comparing
instructions, for two paths to arise with equal probability in some
probability space. Whereas we begin with instructions and seek paths,
Diaconis and Freedman begin with paths and consider instructions.

The following proposition asserts the necessity of the Crossings
property in Theorem \ref{thm:array_to_walk}.

\begin{Prop}\label{prop:up_down_counts}
 For any walk $w$ with start, terminal, and maximum levels $\cS$,
 $\cT$, and $\cL$ respectively, and for any $i\in [0,\cL+1]$,
 \begin{align}\label{eq:up_down_counts}
  u^{w}_{i-1} - d^{w}_i = \cf\{i\leq \cT\} - \cf\{i\leq \cS\},
 \end{align}
 where we define $u^w_{-1} = d^w_{\cL+1} = 0$.
\end{Prop}

\begin{proof}
 Consider the behavior of a walk $w$ around one of its levels $i$. The
 walk may be decomposed into: (i) an initial approach to level $i$
 (trivial when $i$ is the start level), (ii) several excursions above
 and below $i$, and (iii) a final escape from $i$ (trivial when $i$ is
 the terminal level). Such a decomposition is shown in Figure
 \ref{fig:level_decomp}.

 The down-crossing count $d_i$ must equal the number of excursions
 below level $i$, plus 1 if the terminal level is (and final escape
 goes) strictly below level $i$. Similarly, $u_{i-1}$ must equal the
 number of excursions below $i$, plus 1 if the start level is (and
 thus the initial approach comes from) strictly below level $i$.
\end{proof}

\begin{figure}[htb]\centering
 \input{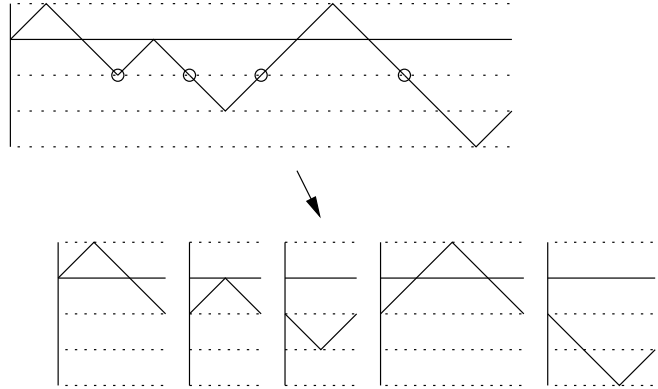}
 \caption{A walk decomposed into an initial approach to a level, excursions from that level, and a final escape.}\label{fig:level_decomp}
\end{figure}

We observe several special cases of this formula.
\begin{Cor}
 \begin{enumerate}
 \item If $w$ is a bridge then $u^w_i = d^w_{i+1}$ for each $i$.
 \item The down-crossing count $d^w_0 = 0$ unless $w$ is a
   first-passage bridge to a negative value, in which case $d^w_0 =
   1$.
 \item The up-crossing count $u^w_{\cL} = 0$ unless $w$ is a
   first-passage bridge to a positive value, in which case $u^w_{\cL}
   = 1$.
 \end{enumerate}
\end{Cor}

We prove the sufficiency of the Bookends and Crossings properties
for Theorem \ref{thm:array_to_walk} by structural induction within
certain equivalence classes of marked arrays.

\begin{Def}
 We say that two marked arrays are \emph{similar}, denoted
 $(\bx,\cP)\sim (\bx',\cP')$, if: (1) $\cP = \cP'$, (2) $u^{\bx}_i =
 u^{\bx'}_i$ and $d^{\bx}_i = d^{\bx'}_i$ for each $i$, and (3) the
 final increment of each row of $\bx$ equals the final increment of
 the corresponding row of $\bx'$.
\end{Def}

This equivalence relation corresponds to a relation between paths
observed in Diaconis and Freedman\cite{DiacFree80}. Note that
similarity respects both the Bookends and Crossings properties. The
following is the base case for our induction.

\begin{Lem}\label{lem:array_base_case}
 Suppose that $(\bx,\cP)$ is a valid marked array with the property
 that, within each row of $\bx$, all but the final increment are
 arranged with all down-steps preceding all up-steps. Then there
 exists a walk $w$ corresponding to $(\bx,\cP)$.
\end{Lem}

We sketch a proof with two observations. Firstly, the proof of this
lemma follows along the lines of the proof of Proposition
\ref{prop:up_down_counts}. Secondly, the corresponding walk $w$ would
be of the form: (1) an initial direct descent from start level to
minimum (except in the case $\cT = -1$, for which this descent may not
reach the minimum) followed by (2) an up-down sawing pattern between
the levels 0 and 1, and then between levels 1 and 2, on up to levels
$\cL-1$ and $\cL$, and finally (3) a direct descent from the maximum
level $\cL$ to the terminal level $\cT$ (except in the case $\cT =
\cL+1$, for which this descent is replaced by a single, final
up-step). A walk of this general form is shown in Figure
\ref{fig:MA_base_case}.

\begin{figure}[htb]\centering
 \includegraphics{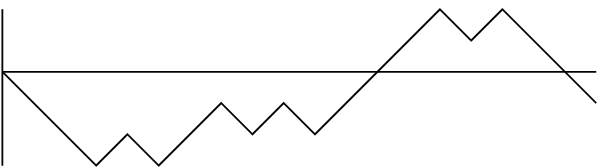}
 \caption{A walk corresponding to an array of the form described in Lemma \ref{lem:array_base_case}.}\label{fig:MA_base_case}
\end{figure}

We follow with the remainder of our induction argument.

\begin{proof}[Proof of Theorem \ref{thm:array_to_walk}]
 The necessity of the Bookends property is clear, and that of the
 Crossings property is asserted in Proposition
 \ref{prop:up_down_counts}. If there exists a walk corresponding to a
 given marked array then its uniqueness is clear from the algorithm
 stated in Figure \ref{fig:reconst_alg}. So it suffices to prove that
 for every valid marked array, there exists a corresponding walk. We
 proceed by structural induction within the $\sim$-equivalence
 classes.

 \emph{Base case:}\ Every $\sim$-equivalence class of valid marked
 arrays contains one of the form described in Lemma
 \ref{lem:array_base_case}. Thus, each class contains a marked array
 that corresponds to some walk.

 \emph{Inductive step:}\ Suppose that $(\bx,\cP)$ is a valid marked
 array that corresponds to a walk $w$. Let $\bx'$ denote an array
 obtained by swapping two consecutive, non-final increments within
 some row $\x_i$ of $\bx$, and leaving all other increments in
 place. Operations of this form generate a group action whose orbits
 are the $\sim$-equivalence classes; thus, it suffices to prove that
 $(\bx',\cP)$ corresponds to some walk.

 As in our proof of Proposition \ref{prop:up_down_counts}, we
 decompose $w$ into an initial approach to level $i$, excursions away
 from level $i$, and a final escape.

 Take, for example, the array:
 \begin{equation*}
  \x_0 = (1),\ \x_1 = (1,-1),\ \x_2 = (1,-1,1,-1),\\
  \x_3 = (1,-1,-1,1,-1),\ \x_4 = (-1,-1),
 \end{equation*}
 with $\cP = 0$. This corresponds to the walk $w$ shown in Figure \ref{fig:level_decomp}.
 $$\begin{array}{r|r|r|r|r|r|r|r|r|r|r|r|r|r|r|r}
    i     & 0&1&2&3&4&5&6&7&8&9&10&11&12&13&14\\\hline
    w(i)  & 0&1&0&-1&0&-1&-2&-1&0&1&0&-1&-2&-3&-2
   \end{array}$$
 Suppose that $\bx'$ is formed by swapping two consecutive increments
 within $\x_2$. Then we decompose the values of $w$ around level $2$,
 which corresponds to the value $w(j) = -1$:
 $$(0,1,0);\ (-1,0);\ (-1,-2);\ (-1,0,1,0);\ (-1,-2,-3,-2).$$ This is
 analogous to the decomposition depicted in Figure
 \ref{fig:level_decomp}. The three middle blocks are excursions.

 The non-final increments of $\x_i$ are the initial increments of
 excursions of $w$ away from level $i$ (in the special case $i= \cT$,
 the final increment of $\x_i$ also begins an excursion). Each $1$
 corresponds to an excursion above level $i$, and each $-1$ to an
 excursion below. In the example, the $(1,-1,1)$ that appear before
 the final increment of $\x_2$ correspond to the three excursions
 mentioned above. Swapping a consecutive `+1' and `-1' in $\x_i$ while
 leaving the $(\x_j)_{j\not= i}$ untouched corresponds to swapping a
 consecutive upward and downward excursion.

 Returning to the example, swapping the second and third increments in
 $\x_2$ corresponds to swapping the second and third excursions of $w$
 away from the value $-1$, resulting in the value sequence:
 $$(0,1,0);\ (-1,0);\ (-1,0,1,0);\ (-1,-2);\ (-1,-2,-3,-2).$$ Because
 the middle three blocks all begin at the value $-1$ and end adjacent
 to it, swapping two of these  result in the value sequence for a
 walk $w'$ -- that is, a sequence of values starting at 0, and with
 consecutive differences of $\pm 1$. Thus, there exists a walk $w'$
 corresponding to $(\bx',\cP)$.
\end{proof}

Theorem \ref{thm:array_to_walk} may be generalized to classify
instruction sets for walks on directed multigraphs. In that setting
the Crossings property is replaced by a condition along the lines of
``in-degree equals out-degree,'' and the Bookends property is replaced
by a condition resembling ``the last-exit edges from each visited,
non-terminal vertex form a directed tree.'' The latter of these has
been observed by Broder\cite{Broder89} and Aldous\cite{Aldous90} in
their study of an algorithm to generate random spanning trees. See
also \cite[p.\ 12]{Bollobas98}.

We now digress from our main thread of proving the bijection between
walks and quantile pairs to address the question: given a valid array
$\bx$, what can we say about the indices $\cP$ for which $(\bx,\cP)$
is valid? We begin by asking: what does the Bookends property look
like?

By the definition of $\cT$ given in \eqref{eq:special_rows}, it must
differ from $\cP$ by exactly 1. Therefore the two classifications $i
\leq \min\{\cP,\cT\}$ and $i \geq \max\{\cP,\cT\}$ are exhaustive and
non-intersecting. Given $\bx$, there exists a $\cP$ for which the
Bookends property is satisfied if and only if, for all $i$ below a
certain threshold $\x_i$ ends in an up-step, and for all $i$ above
that threshold $\x_i$ ends in a down-step; if this is the case then
$\cP$ and $\cT$ must stand on either side of that threshold.

Consider the following array.
\begin{align*}
 \left.\begin{array}{r}
  \x_4 = (-1)\\
  \x_3 = (+1,-1,-1)\\
 \end{array}\right\} &\\
 \left.\begin{array}{r}
  \x_2 = (-1,+1,+1,-1,+1)\\
  \x_1 = (+1,+1,-1,+1)\\
  \x_0 = (+1)
 \end{array}\right\} &
\end{align*}
The row-ending increments transition from $1$s to $-1$s between rows 2
and 3. Thus, the Bookends property requires that either $\cP =2$ and
$\cT = 3$ or vice versa. Both of these choices are consistent with
equation \eqref{eq:special_rows}.

\begin{Prop}\label{prop:Bookends_P}
 Given an increment array $\bx$, there are at most two distinct
 triples $(\cP,\cT,\cS)$ that satisfy: (i) equation
 \eqref{eq:special_rows}, (ii) the Bookends property, and (iii) the
 property $\cP\in [0,\cL]$. Furthermore, if there are two such triples
 then no entry is the same in both triples.
\end{Prop}

\begin{proof}
 We begin with the special cases corresponding to first-passage
 bridges. First, suppose that every row of $\bx$ ends in a `1'. Then
 the Bookends property and the bounds on $\cP$ are only satisfied if
 $\cP = \cL$, and then $\cT$ and $\cS$ are pinned down by
 \eqref{eq:special_rows}; in particular $\cT = \cL+1$. By a similar
 argument, if every row ends in a `-1' then $\cP$ must equal 0, and
 again $\cT$ and $\cS$ are specified by \eqref{eq:special_rows} with
 $\cT = -1$.

 Now suppose that some rows of $\bx = (\x_i)_{i=0}^{\cL}$ end in `1's
 and others in `-1's. Then there exists a $\cP$ for which the Bookends
 property is satisfied if and only if there is some number $a\in
 [0,\cL)$ such that, for $i\leq a$ row $\x_i$ ends in a `1', and for
   $i > a$ row $\x_i$ ends in a `-1'. So the Bookends property and
   \eqref{eq:special_rows} force $(\cP,\cT)$ to equal either $(a,a+1)$
   or $(a+1,a)$. Thus, the two triples which satisfy all three
   properties are
 \begin{align}
  (\cP,\cT,\cS) = (a,a+1,a+1-\sigma_{\bx})\rm{\; or\; }(a+1,a,a-\sigma_{\bx}).
 \end{align}
\end{proof}

We can now classify with which $\cP$ a given $\bx$ may form a valid
marked array.

\begin{Thm}\label{thm:array_to_walk_multiplicity}
 Let $\bx = (\x_i)_{i=0}^{\cL}$ be a valid array. If
 $\sigma_{\bx}\not= 0$ then $\bx$ corresponds to a unique walk, and if
 $\sigma_{\bx} = 0$ then $\bx$ corresponds to exactly two distinct
 bridges.
\end{Thm}

\begin{proof}
 By the uniqueness asserted in Theorem \ref{thm:array_to_walk} it
 suffices to prove that if $\sigma_{\bx} \not= 0$ (or if $\sigma_{\bx}
 = 0$) then there is a unique $\cP$ (respectively exactly two distinct
 values $\cP$) for which $(\bx,\cP)$ is valid. We proceed with three
 cases.

 \emph{Case 1: $\sigma_{\bx} > 0$}. By Theorem
 \ref{thm:array_to_walk}, for any valid choice of $\cP$ the resulting
 $\cS$ lies within $[0,\cL]$ -- a walk must start at a level from
 which it has some increments. By the Crossings property,
 \begin{align}
  u_{i-1} = d_{i}\rm{\ for\ }i \leq \cS\rm{\ and\ }u_{\cS+1} = d_{\cS+1} + 1.
 \end{align}
 These two properties uniquely specify $\cS$; and by Proposition
 \ref{prop:Bookends_P} our choice of $\cS$ uniquely specifies $\cP$.

 \emph{Case 2: $\sigma_{\bx} < 0$}. This dual to case 1. In this case,
 $\cS$ must satisfy
 \begin{align}
  u_{i} = d_{i+1}\rm{\ for\ }i \geq \cS\rm{\ and\ }u_{\cS-1} = d_{\cS} - 1.
 \end{align}
 Again $\cS$ is uniquely specified, and by Proposition
 \ref{prop:Bookends_P} $\cP$ is uniquely specified.

 \emph{Case 3: $\sigma_{\bx} = 0$}. In this case, the Crossings
 property asserts that $u_i = d_{i+1}$ for every $i$; this places no
 constraints on $\cP$, $\cT$, or $\cS$. By our assumption that $\bx$
 is valid, it therefore satisfies the crossings property regardless of
 $\cP$, so the only constraints on $\cP$ are coming from the Bookends
 property.

 The Crossings property tells us that
 $$d_0 = u_{-1} = 0\rm{\;\; and\;\;}u_{\cL} = d_{\cL+1} = 0,$$ so
 $\x_0$ ends in a `1' and $\x_{\cL}$ ends in a `-1'. We observed in
 the proof of Proposition \ref{prop:Bookends_P} that in this case
 there are either zero or two values $\cP$ for which $(\bx,\cP)$ is
 valid. And by our assumption that $\bx$ is valid there are two such
 values.
\end{proof}

\section{Partitioned walks}
\label{sec:saw}

In this section we introduce partitioned walks and define the map $\b$
suggested in equations \eqref{eq:q_decomp} and \eqref{eq:q_decomp2}. A
partitioned walk is a walk with its increments partitioned into
contiguous blocks with one block distinguished. Partitioned walks
correspond in a natural manner with marked arrays (not just valid
marked arrays). Theorem \ref{thm:saw_to_walk}, which is the main
result of this section, describes the $\b$-image of the valid marked
arrays. The elements of this image set are called \emph{quantile
partitioned walks}. In section \ref{sec:bijection_pf} we demonstrate
a bijection between the quantile partitioned walks and the quantile
pairs.

Let $w$ be a walk of length $n$, and let the $u^w_i$ and $d^w_i$ be
the up- and down- crossing counts of $w$ from level $i$, as defined in
the previous section.

\begin{Def}
 \label{def:saw_tooth}
 For $j\in [0,\cL + 1]$, define $t^w_j$ to be the number of
 increments of $w$ at levels below $j$:
 $$t^w_j := \sum_{i=0}^{j-1} u_i + d_i.$$ So $0= t^w_0 < \cdots <
 t^w_{\cL+1} = n$. We call $t^w_j$ the $j^{\rm{th}}$ \emph{saw tooth}
 of $w$.
\end{Def}

Whenever it is clear from context we  suppress the superscripts in
the saw tooth of a walk.

Note that the helper variable employed in the quantile bijection theorem, Theorem \ref{thm:q_bijectiopn}, appears in this sequence:
\begin{align}
 \Qp_w^{-1}(n) = t_{\cP+1}.\label{eq:helper_tooth}
\end{align}
This is because the $n^{\rm{th}}$ increment of $w$ is its final
increment at the preterminal level.

We are interested in the saw teeth in part because less considerations
go into the value of $Q(w)$ at $t^w_j$ than at some general $t$. In
particular, $Q(w)(t^w_j)$ ignores the order of increments within each
level of $w$.

\begin{Lem}\label{lem:Q_at_saw_teeth}
 Let $w$ be a walk with up- and down-crossing counts $(u_i)$ and
 $(d_i)$ and saw teeth $(t_i)$. Let $\cS$, $\cT$, and $\cL$ be the
 start, terminal, and maximum levels of $w$. Then
 \begin{align}
  Q(w)(t_j) &= \sum_{i<j}u_i - d_i\rm{\; for each\; }j\in [0,\cL+1].\label{eq:saw_0}
 \end{align}
 This may be restated in the closed form
 \begin{align}
  Q(w)(t_{j+1}) &= u_j + (j - \cS)_+ - (j - \cT)_+\rm{\; for each\; }j\in [-1,\cL].\label{eq:saw}
 \end{align}
\end{Lem}

\begin{proof}
 We note that $Q(w)(t_j)$ is a sum of all increments of $w$ that
 belong to levels less than $j$. This proves equation
 \eqref{eq:saw_0}. Regrouping the terms of \eqref{eq:saw_0} and
 applying equation \eqref{eq:up_down_counts} then gives equation
 \eqref{eq:saw}.
\end{proof}

Equation \eqref{eq:saw} is a discrete-time form of Tanaka's formula,
the continuous-time version of which we recall in section
\ref{sec:Jeulin}. Briefly, the value $Q(w)(t_{j+1})$ corresponds to
the integral $\int_0^1\cf\{X(t) \leq a\}dX(t)$ in that it sums all
increments of $w$ which appear below the fixed level $j$; the term
$u_j$ corresponds to $\frac12\ell^a$ -- roughly half of the visits of
a simple random walk to level $j$ are followed by up-steps; and the
latter terms $j-\cS$ and $j-\cT$ correspond to $a$ and $a-X(1)$.
Further discussion of the discrete Tanaka formula may be found in
\cite{Kudzma82,CsorReve85_1,Szabados90,SzabSzek09}.

Equation \eqref{eq:saw} takes the following form in the bridge case.

\begin{Cor}
 If $w$ is a bridge then $Q(w)(t_{j+1}) = u_j$ for each $j$.
\end{Cor}

\begin{figure}[htb]\centering
 \input{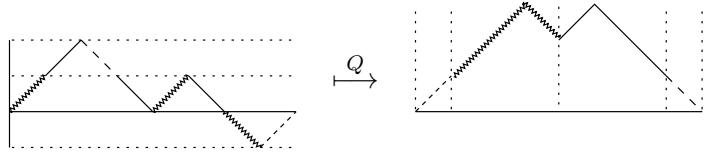}
 \caption{Increments emanating from a common level in $w$ appear in a contiguous block in $Q(w)$.}\label{fig:intro_p_walk}
\end{figure}

The saw teeth partition the increments of $Q(w)$ into blocks in the
manner illustrated in Figure \ref{fig:intro_p_walk}: increments from
the $j^{\rm{th}}$ block, between $t_j$ and $t_{j+1}$, correspond to
increments from the $j^{\rm{th}}$ level of $w$. This partition
provides the link between increment arrays and the quantile
transform. This is illustrated in Figure \ref{fig:saw_to_array}. The
saw teeth are shown as vertical dotted lines partitioning the
increments of $Q(w)$. Each block of this partition consists of the
increments from a row of $\bx_w$, stuck together in sequence.

\begin{figure}[htb]\centering
 \input{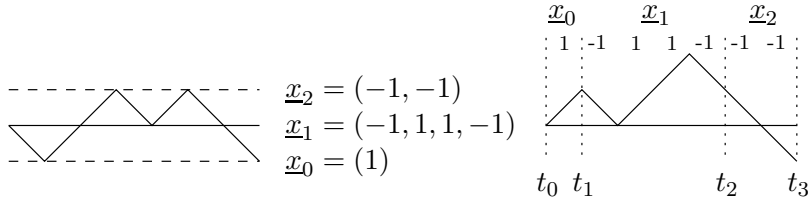}
 \caption{Left to right: a walk, its increment array, and its quantile transform partitioned by saw teeth.}\label{fig:saw_to_array}
\end{figure}

We will now define the map $\b$ alluded to in equations \eqref{eq:q_decomp} and \eqref{eq:q_decomp2} such that it will satisfy
\begin{align}
 \b\circ\a(w) = (Q(w),(t^w_i)_{i=0}^{\cL+1},\cP_w).
\end{align}
We define the \emph{partitioned walks} to serve as a codomain for this
map.

\begin{Def}
 A \emph{partitioned walk} is a triple $\bv =
 (v,(t_i)_{i=0}^{\cL+1},\cP)$ where $v$ is a walk, say of length $n$,
 $$0 = t_0 < t_1 < \cdots < t_{\cL+1} = n,$$ and $\cP\in
 [0,\cL]$. Here we are taking the $t_j$, $\cL$, and $\cP$ to be
 arbitrary numbers, rather than the saw teeth and distinguished levels
 of $v$. The name ``partitioned walk'' refers to the manner in which
 the times $t_i$ partition the increments of $v$ into blocks. We call
 the block of increments of $v$ bounded by $t_{\cP}$ and $t_{\cP+1}$
 the \emph{preterminal block} of $\bv$. We say that such a partitioned
 walk $\bv$ \emph{corresponds to a walk $w$} if $\bv =
 (Q(w),(t^w_i)_{i=0}^{\cL_w},\cP_w)$.
\end{Def}

\begin{Def}
 \label{def:beta_gamma}
 Define $\b$ to be the map which sends a marked array
 $((\x_i)_{i=0}^{\cL},\cP)$ to the unique partitioned walk
 $(v,(t_i)_{i=0}^{\cL+1},\cP)$ which satisfies
 \begin{align}
  \x_i = \left(\begin{array}{r}
 	v(t_i+1)-v(t_i),\; v(t_{i}+2)-v(t_i+1),\\
  		{\cdots},\; v(t_{i+1}) - v(t_{i+1}-1)
  		\end{array}\right)\rm{\ for every\;}i\in [0,\cL].\label{eq:p_walk_to_array}
 \end{align}
 Define $\g$ to be the map from partitioned walks to walk-index pairs
 given by
 \begin{equation}
  \g(v,(t_i),\cP) := (v,t_{\cP+1}).
 \end{equation}
\end{Def}

We address the map $\g$ in section \ref{sec:bijection_pf}. The map
$\b$ may be thought of as stringing together increments one row at a
time, as illustrated on the right in Figure \ref{fig:saw_to_array}, as
well as in Figure \ref{fig:p_walk_to_array}. In this latter example
neither the array nor the partitioned walk corresponds to any
(unpartitioned) walk.

\begin{figure}[htb]\centering
 $\displaystyle\left\{\begin{array}{l}
 		 \x_4 = (-1)\\
 		 \x_3 = (1,1,-1)\\
 		 \x_2 = (1,1)\\
 		 \x_1 = (1,-1,-1)\;\;\bigstar\\
 		 \x_0 = (-1,-1)
 		\end{array}\right.\;\;\;\longleftrightarrow$\;\;\;\raisebox{-.5\height}{\input{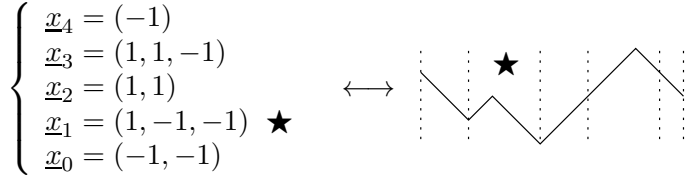}}
 \caption{A marked array and its image under $\b$.}\label{fig:p_walk_to_array}
\end{figure}

While it is clear that $\b$ is a bijection, we are particularly
interested in the image of the set of \emph{valid} marked
arrays. Before we describe this image, we make a couple more
definitions.

\begin{Def}\label{def:p_walk_trough}
 Let $\bv = (v,(t_i)_{i=0}^{\cL+1},\cP)$ be a partitioned
 walk. Motivated by the later terms in equation \eqref{eq:saw} we
 define the \emph{trough function} for $\bv$ to be
 \begin{align}
  M_{\bv}(j) := (j - \cS)_+ - (j - \cT)_+,\label{eq:p_walk_trough}
 \end{align}
 where we define the indices $\cT$ and $\cS$ via
 \begin{align}
  \cT := \cP + v(t_{\cP+1}) - v(t_{\cP+1}-1)\rm{,\;\; and\;\; }\cS := \cT - v(t_{\cL+1}).\label{eq:special_blocks}
 \end{align}
 This is the partitioned walk analogue to equation
 \eqref{eq:special_rows} for marked arrays. If they exist, then we
 call the block of increments bounded by $t_{\cS}$ and $t_{\cS+1}$ the
 \emph{start block}, and the block bounded by $t_{\cT}$ and
 $t_{\cT+1}$ the \emph{terminal block}.
\end{Def}

\begin{Def}
  A partitioned walk has the \emph{Bookends property} if for $i \leq
  \cT, \cP$, the $t_{i+1}^{\rm{st}}$ increment of $v$ (i.e.\ the last
  increment of the $i^{\rm{th}}$ block) is an up-step; likewise, if $i
  \geq \cT, \cP$, then the $t_{i+1}^{\rm{st}}$ increment of $v$ is a
  down-step.

  A partitioned walk has the \emph{Saw property} if for each $j\in
  [0,\cL]$,
  \begin{align}
   v(t_{j+1}) + v(t_j) = t_{j+1} - t_j + 2M_{\bv}(j).\label{eq:saw_property_v}
  \end{align}

 A partitioned walk with the Bookends and Saw properties is called a
 \emph{quantile partitioned walk}.
\end{Def}

\begin{Thm}\label{thm:saw_to_walk}
 The map $\b$ bijects the set of valid marked arrays with the set of
 quantile partitioned walks.
\end{Thm}

The equivalence of the Bookends properties for partitioned walks
versus marked arrays is clear. We first define the saw path of a
partitioned walk, and we use this to generate several useful
restatements of the Saw property. Then we demonstrate the equivalence
of the Saw property of partitioned walks to the Crossings property of
arrays, and use this to prove the theorem.

\begin{Def}
 For any partitioned walk $\bv = (v,(t_i),\cP)$, we define the
 \emph{saw path} $S_{\bv}$ to be the minimal walk that equals $v$ at
 each time $t_i$.
\end{Def}

The saw teeth of a walk $w$ have been so-named because they typically
coincide with the maxima of the saw path $S_{\bv}$, where $\bv =
\b\circ\a(w)$.

\begin{Lem}\label{lem:trough_bounds_saw}
 Let $\bv$ be a partitioned walk, and let $u_j$ and $d_j$ denote the
 number of up- and down-increments of $v$ between times $t_j$ and
 $t_{j+1}$ for each $j$. Then the saw property for $\bv$ is equivalent
 to each of the following families of equations. For every
 $j\in[0,\cL]$,
 \begin{align}
  M_{\bv}(j) &= -d_j + \sum_{i<j} u_i - d_i\rm{, or equivalently}\label{eq:trough_from_crossings}\\
  M_{\bv}(j) &= \min_{t\in [t_j,t_{j+1}]}S_{\bv}(t).\label{eq:trough_bounds_saw}
 \end{align}
\end{Lem}

\begin{proof}
 By definition of the saw path
 \begin{align}
  \min_{t\in [t_j,t_{j+1}]}S_{\bv}(t) = -d_j + \sum_{i<j} u_i - d_i.
 \end{align}
 Thus, it suffices to show that the saw property is equivalent to
 \eqref{eq:trough_from_crossings}.

 First we express a few quantities in terms of the $u_i$s and $d_i$s:
 \begin{align}
  t_{j+1} - t_j &= u_j + d_j,\label{eq:u_d_time_diff}\\
  v(t_{j+1}) - v(t_j) &= u_j - d_j\rm{, and}\label{u_d_value_diff}\\
  v(t_j) &= \sum_{i<j} u_i - d_i.\label{eq:v_from_crossings}
 \end{align}
 From these equations we obtain
 \begin{align*}
  v(t_{j+1}) + v(t_j) - (t_{j+1} - t_j) = - 2d_j + 2v(t_j) = -2d_j + 2\sum_{i<j} u_i - d_i.
 \end{align*}
 The saw property asserts that $2M_{\bv}(j)$ equals the expression on
 the left-hand side above. The claim follows.
\end{proof}

Figure \ref{fig:saw_down} shows two examples of
$$w \stackrel{\b\circ\a}{\longmapsto} (Q(w),(t^w_i),\cP).$$ The saw
teeth are represented by vertical dotted lines and the preterminal
block is starred. The saw path is drawn in dashed lines where it
deviates below $Q(w)$. In between each pair of teeth $t_j$ and
$t_{j+1}$ we show a horizontal dotted line at the level of
$M_{\bv}(j)$. Observe how the saw path bounces off of these horizontal
lines; this illustrates equation \eqref{eq:trough_bounds_saw}.

\begin{figure}[tbh]
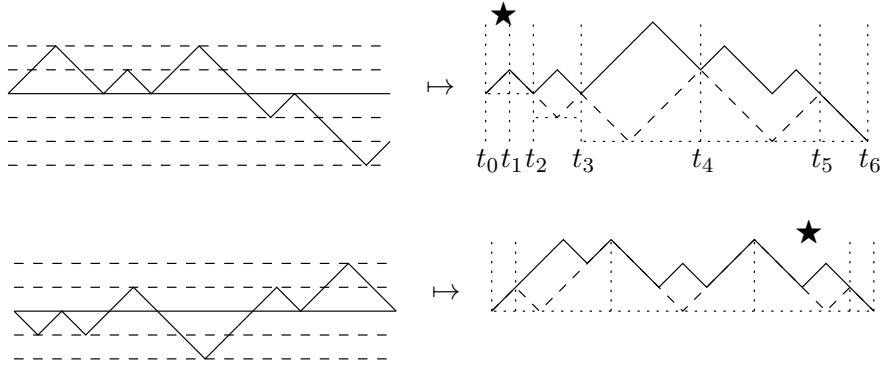
\centering
 \input{saw_down.pstex_t}\\[12pt]
 \input{saw_br.pstex_t}
 \caption{Two walks and their quantile transforms overlayed with saw teeth, saw paths, and troughs.}\label{fig:saw_down}
\end{figure}

In Figure \ref{fig:saw_bad_troughs} we show the saw path of a
partitioned walk $\bv$ which doesn't have Saw property. This diagram
follows the same conventions as the diagrams on the right hand side in
Figure \ref{fig:saw_down}.

\begin{figure}[htb]\centering
 \input{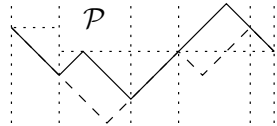}
 \caption{A general partitioned walk and its saw path.}\label{fig:saw_bad_troughs}
\end{figure}

By definition of the saw path
\begin{align}
 v(t)\geq S_{\bv}(t)\rm{\ for every\ }t.
\end{align}
This gives us the following corollary to Lemma
\ref{lem:trough_bounds_saw}.

\begin{Cor}\label{cor:trough_bound}
 If $\bv = (v,(t_i),\cP)$ is a partitioned walk with the Saw property
 then for $t \in [t_j,t_{j+1}]$,
 \begin{align}\label{eq:trough_bound}
  v(t)\geq M_{\bv}(j).
 \end{align}
\end{Cor}

\begin{Lem}\label{lem:S_in_bounds_p_walk}
 If $\bv = (v,(t_j)_{j=0}^{\cL+1},\cP)$ is a partitioned walk with the Saw property then the index $\cS$ of its start block falls within $[-1,\cL+1]$.
\end{Lem}

\begin{proof}
 We consider three cases.

 \emph{Case 1:}\ $v(n) = 0$. Then $\cS = \cT$, and so the desired
 result follows from the definition of $\cT$ in
 \eqref{eq:special_blocks}, and from the property $\cP\in [0,\cL]$
 which is stipulated in the definition of a partitioned walk.

 \emph{Case 2:}\ $v(n) > 0$. Then $\cS < \cT \leq \cL+1$. But if $\cS
 < -1$ then $M(0) > 0$. This would contradict Corollary
 \ref{cor:trough_bound} at $j=0$, $t=0$.

 \emph{Case 3:}\ $v(n) < 0$. Then $\cS > \cT \geq -1$. If both
 $\cS,\cT > \cL$ then $M(\cL) = 0$; this would contradict Corollary
 \ref{cor:trough_bound} at $j=\cL$ with $t=n$. And if $\cT\leq \cL <
 \cS$ then
 \begin{align*}
  M(\cL) > (\cL - \cS) - (\cL - \cT) = v(n),
 \end{align*}
 which would again contradict Corollary \ref{cor:trough_bound} at the
 same point.
\end{proof}

In fact, it follows from Theorem \ref{thm:saw_to_walk} that
$\cS\in[0,\cL]$, but we require the weaker result of Lemma
\ref{lem:S_in_bounds_p_walk} to prove the theorem.

\begin{proof}[Proof of Theorem \ref{thm:saw_to_walk}]
 Let $(\bx,\cP)$ be a marked array and let $\b(\bx,\cP) = \bv =
 (v,(t_j)_{j=0}^{\cL+1},\cP)$. Clearly $(\bx,\cP)$ has the Bookends
 property for arrays if and only if $\bv$ has the Bookends property
 for partitioned walks. For the remainder of the proof, we assume that
 both have the Bookends property.

 It suffices to prove that $\bv$ has the Saw property if and only if
 $(\bx,\cP)$ has the Crossings property. In fact, the Saw property is
 equivalent to the Crossings property even outside the context of the
 Bookends property, but we sidestep that proof for brevity's sake.

 Let $(u_j)$ and $(d_j)$ denote the up- and down-crossing counts of
 $\bx$; these also count the up- and down-steps of $v$ between
 consecutive partitioning times $t_j$ and $t_{j+1}$. Let $\cS$ and
 $\cT$ denote the start and terminal row indices for $(\bx,\cP)$, or
 equivalently, the start and terminal block indices for $\bv$.

 The Saw property for $\bv$ is equivalent to the following three
 conditions:
 \begin{align}
  M_{\bv}(-1) &= 0,\;\; M_{\bv}(\cL+1) = v(n)\rm{, and}\label{eq:trough_extremes}\\
  M_{\bv}(j) - M_{\bv}(j-1) &= u_{j-1} - d_j\rm{ for each }j\in [0,\cL+1].\label{eq:Saw_vs_Crossings}
 \end{align}
 The Saw property implies \eqref{eq:trough_extremes} by way of Lemma
 \ref{lem:S_in_bounds_p_walk}; and given \eqref{eq:trough_extremes},
 equation \eqref{eq:Saw_vs_Crossings} is equivalent to
 \eqref{eq:trough_from_crossings}, which in turn is equivalent to the
 Saw property by Lemma \ref{lem:trough_bounds_saw}.

 The Crossings property for $(\bx,\cP)$ is equivalent to those same
 three conditions. The validity of $(\bx,\cP)$ implies
 \eqref{eq:trough_extremes} via Theorem \ref{thm:array_to_walk}:
 because the array corresponds to a walk, it must have $\cS\in
 [0,\cL]$. Furthermore, given \eqref{eq:trough_extremes} the Crossings
 property may be shown to be equivalent to \eqref{eq:Saw_vs_Crossings}
 by substituting in the formula \eqref{eq:p_walk_trough} for
 $M_{\bv}$.
\end{proof}

\section{The quantile bijection theorem}
\label{sec:bijection_pf}

In this section we give a lemma which will help us show that $\g$ is
injective on the quantile partitioned walks. We then apply this lemma
to prove Theorem \ref{thm:q_bijection}, the Quantile bijection theorem.

\begin{Lem}\label{lem:saw_geom}
 A partitioned walk $\bv = (v,(t_i)_{i=0}^{\cL+1},\cP)$ has the Saw
 and Bookends properties if and only if the following two conditions
 hold.
 \begin{enumerate}
  \item For every $j\in [0,\cP]$
  \begin{align}
   t_j = \inf\{t\geq 0 : v(t) = t_{j+1} - t + 2M_{\bv}(j) - v(t_{j+1})\}.\label{eq:saw_left}
  \end{align}
  \item For every $j\in [\cP+1,\cL]$
  \begin{align}
   t_{j+1} = \inf\{t\geq 0 : v(t) = t - t_j + 2M_{\bv}(j) - v(t_j)\}.\label{eq:saw_right}
  \end{align}
 \end{enumerate}
\end{Lem}

\begin{proof}
 The Saw property of $\bv$ is equivalent, by algebraic manipulation,
 to the conditions that for $j\in [0,\cP]$, the $t_j$ must solve
 \begin{align}
  v(t) + t = t_{j+1} + 2M_{\bv}(j) - v(t_{j+1})\label{eq:saw_left_0}
 \end{align}
 for $t$, and for $j\in [\cP+1,\cL]$, the $t_{j+1}$ must solve
 \begin{align}
  v(t) - t = - t_j + 2M_{\bv}(j) - v(t_j).\label{eq:saw_right_0}
 \end{align}

 Now suppose that some $s$ solves equation \eqref{eq:saw_left_0} for
 some $j\leq \cP$. A time $r < s$ offers another solution to
 \eqref{eq:saw_left_0} if and only if
 $$v(r) + r = v(s) + s.$$ This is equivalent to the condition that $v$
 takes only down-steps between the times $r$ and $s$. Therefore $t_j$
 equaling the least solution to \eqref{eq:saw_left_0} is equivalent to
 the $t_j^{\rm{th}}$ increment of $v$ being an up-step, as required by
 the Bookends property.

 Similarly, suppose that $s$ solves equation \eqref{eq:saw_right_0}
 for some $j\geq \cP+1$. A time $r<s$ provides another solution if and
 only if
 $$v(r) - r = v(s) - s,$$ which is equivalent to the condition that
 $v$ takes only up-steps between $r$ and $s$. Therefore $t_j$ equaling
 the least solution to \eqref{eq:saw_right_0} is equivalent to the
 $t_{j+1}^{\rm{st}}$ increment of $v$ being a down-step, as required
 by the Bookends property.

 Equation \eqref{eq:special_blocks} defines $\cT$ from $\cP$ in such a
 way that the $t_{\cP+1}^{\rm{st}}$ increment of $v$ will always
 satisfy the Bookends property. Thus, if \eqref{eq:saw_left} holds for
 $j\in [0,\cP]$ and \eqref{eq:saw_right} holds for every $j\in
 [\cP+1,\cL]$, then the Bookends property is met at every $t_j$.
\end{proof}

Finally, we are equipped to prove our main discrete-time result.

\begin{proof}[Proof of the Quantile bijection, Theorem \ref{thm:q_bijection}.]
 Definitions \ref{def:alpha} and \ref{def:beta_gamma} define the maps
 $\a$, $\b$, and $\g$ in such a way that, for a walk $w$ of length
 $n$,
 \begin{align*}
  \g\circ\b\circ\a(w) = (Q(w),\Qp_w^{-1}(n)).
 \end{align*}
 Theorem \ref{thm:q_non_simple} asserts that this map sends walks to quantile pairs, and by
 Proposition \ref{prop:quantile_lvl_count} the set of walks with a
 given number of up- and down-steps has the same cardinality as the
 set of quantile pairs with those same numbers of up- and down-steps.
 Theorems \ref{thm:array_to_walk} and \ref{thm:saw_to_walk} assert that
 that $\b\circ\a$ bijects the walks with the quantile partitioned
 walks, so it suffices to prove that $\g$ is injective on the quantile
 partitioned walks.

 Now suppose that $\g(\bv) = \g(\bv') = (v,k)$ for some pair of
 quantile partitioned walks $\bv = (v,(t_i)_{i=0}^{\cL+1},\cP)$, and
 $\bv' = (v,(t_i')_{i=0}^{\cL'+1},\cP')$. We define
 \begin{align}
  \tM (i) := (i + v(n) - y_k)_+ - (i - y_k)_+\rm{,\ where\ }y_k = v(k) - v(k-1).\label{eq:trough_from_P}
 \end{align}
 Note that, by definition \ref{def:p_walk_trough},
 \begin{align}
  \tM(i) = M_{\bv}(\cP+i) = M_{\bv'}(\cP'+i)\rm{\ for every\ }i.
 \end{align}

 We prove by induction that $\bv$ must equal $\bv'$, and therefore
 that $\g$ is injective on the quantile partitioned walks.

 \emph{Base case:} $t_{\cP+1} = t'_{\cP'+1} = k$.

 \emph{Inductive step:} We assume that $t_{\cP+1-i} = t'_{\cP'+1-i} >
 0$ for some $i\geq 0$. Then by Lemma \ref{lem:saw_geom}
 \begin{align*}
   t_{\cP - i} = t'_{\cP' - i} = \inf\{t\geq 0 : v(t) = t_{\cP+1-i} - t + 2\tM(-i) - v(t_{\cP+1-i})\}.
 \end{align*}
 Likewise, if we assume $t_{\cP+1+i} = t_{\cP'+1+i}$ for some $i\geq
 0$ then by Lemma \ref{lem:saw_geom},
 \begin{align*}
  t_{\cP+2+i} = t'_{\cP'+2+i} = \inf\{t\geq 0 : v(t) = t - t_{\cP+1+i} + 2\tM(i+1) - v(t_{\cP+1+i})\}.
 \end{align*}

 By induction, $t_{\cP+i} = t'_{\cP'+i}$ wherever both are
 defined. Thus there is some greatest index $I\leq 0$ at which these
 simultaneously reach 0. This $I$ must equal both $-\cP$ and
 $-\cP'$. By the same reasoning $\cL = \cL'$. We conclude that $\bv =
 \bv'$.
\end{proof}

We also have the following special case.

\begin{Cor}\label{cor:quantile_br}
 The quantile transform of a bridge is a Dyck path. Moreover, for a
 uniform random bridge $b$ of length $2n$ and a fixed Dyck path $d$ of
 the same length,
 $$\Pr\{Q(b) = d\} = \frac{2k}{\binom{n}{2}},$$
 where $2k$ is the duration of the final excursion of $d$.
\end{Cor}

\section{The Vervaat transform of a simple walk}
\label{sec:Vervaat}

The quantile transform has much in common with the (discrete) Vervaat
transform $V$, studied in \cite{Vervaat79}. For discussions of this
and related transformations, see Bertoin\cite{Bertoin93} and
references therein. Like the quantile transform, the Vervaat transform
permutes the increments of a walk.

Breaking with usual conventions, let $\rm{mod}\ n$ to denote the map
from $\mathbb{Z}$ to the $(\rm{mod }n)$ representatives $[1,n]$
(instead of the standard $[0,n-1]$).

\begin{Def}
 Given a walk $w$ of length $n$, let
 \begin{align}
  \tau_V(w) = \min\{j\in [0,n] : w(j) \leq w(i)\rm{ for all }i\in [0,n]\}.
 \end{align}
 The \emph{Vervaat permutation} $\Vp_w$ is the cyclic permutation
 $i\mapsto i+\tau_V(w)\ \rm{mod}\ n$. As with the quantile transform,
 we define the \emph{Vervaat transform} $V$ by
 \begin{align}
  V(w)(j) = \sum_{i=1}^j x_{\Vp_w(i)}.
 \end{align}
\end{Def}

Compare this to definition \ref{def:Vervaat_cts}. An example of the
Vervaat transform appears in Figure \ref{fig:vervaat_eg}.

\begin{figure}[htb]\centering
 \input{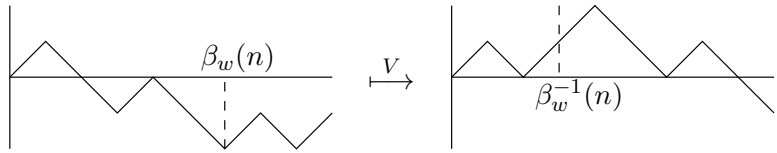}
 \caption{A walk transformed by $V$.\label{fig:vervaat_eg}}
\end{figure}

This transformation was studied by Vervaat because of its asymptotic
properties. As scaled simple random walk bridges converge in
distribution to Brownian bridge, the Vervaat transform of these
bridges converges in distribution to a continuous-time version of the
Vervaat transform, applied to the Brownian bridge.

Surprisingly, the discrete Vervaat transform has a very similar
bijection theorem to that for $Q$.

\begin{Def}
  A \emph{Vervaat pair} is a pair $(v,k)$ where $v$ is a walk of
  length $n$ and $k$ is a nonnegative integer such that $v(j)\geq 0$
  for $0\leq j \leq k$ and $v(j) > v(n)$ for $k\leq j < n$.
\end{Def}

\begin{Thm}\label{thm:Vervaat_wk}
 The map $w \mapsto (V(w), n-\tau_V(w))$ is a bijection between the
 walks of length $n$ and Vervaat pairs.
\end{Thm}

\begin{proof}
 If we know that a pair $(v,k)$ arises in the image of $(V,K)$, then
 it is clear how to invert this map: let $y_i = v(i) - v(i-1)$ for
 each $i$; let $x_i = y_{i+k}$, where we take these indices $\rm{mod
 }n$; and we define $F(v,k)$ to be the walk with increments
 $x_i$. Then $F(V(w),n-\tau_V(w)) = w$. We  show that for every
 $w$ the pair $(V(w),n-\tau_V(w))$ is a Vervaat pair, and that every
 Vervaat pair satisfies $(v,k) = (V(F(v,k)),n-\tau_V(F(v,k)))$.

 Let $w$ be a walk of length $n$. By definition of $\tau_V$, for every
 $j\in [0,\tau_V(w))$ we have $w(j) > w(\tau_V(w))$. It follows that
   $V(w)(j) > v(n)$ for $j\in [n-\tau_V(w),n)$. Likewise, for $j\in
     [\tau_V(w),n]$ we have $w(j) \geq w(\tau_V(w))$; so it follows
     that $V(w)(j) \geq 0$ for $j\in [0,n-\tau_V(w)]$.

 Now, consider a Vervaat pair $(v,k)$. Then by definition of $F$ and
 by the properties of the pair, for $j\in[0,n-k)$ we have $F(v,k)(j) >
   F(v,k)(n-k)$, and for $j\in[n-k,n]$, we have $F(v,k)(j)\geq
   F(v,k)(n-k)$. Thus, $\tau_V(F(v,k)) = n-k$, and the result follows.
\end{proof}

To our knowledge, this result has not been given explicitly in the
literature. This statement strongly resembles our statement of Theorem
\ref{thm:q_bijection}, but we note two differences. The first is the
helper variable. The helper variable in this theorem equals
$\Vp_w^{-1}(n)$ except in the case where $w$ is a first-passage bridge
to a negative value, in which case $\Vp_w^{-1}(n) = n$ whereas $n -
\tau_w = 0$; in our statement of Theorem \ref{thm:q_bijection}, the
helper always equals $\Qp_w^{-1}(n)$ and may not equal 0. The second
difference is that the value $V(w)(k)$ must be non-negative, whereas
$Q(w)(k)$ may equal $-1$ (see Figure \ref{fig:Q_helper}). Again, this
only affects the case where $w(n) <0$.

\begin{Cor}\label{cor:Vervaat_count}
 Let $v$ be a walk of length $n$ and $k\in [1,n]$. If $v(n)\geq 0$
 then $(v,k)$ is a quantile pair if and only if it is a Vervaat
 pair. And in the case $v(n) < 0$, the pair $(v,k)$ is a quantile pair
 if and only if $(v,k-1)$ is a Vervaat pair. In particular, regardless
 of $v(n)$,
 \begin{align}\label{eq:Vervaat_count}
  \#\{w : V(w) = v\} = \#\{w : Q(w) = v\}.
 \end{align}
\end{Cor}

Equation \eqref{eq:Vervaat_count} is a key result as we pass into the
continuous-time setting.

\section{The quantile transform of Brownian motion}
\label{sec:Jeulin}

Our main theorem in the continuous setting compares the quantile
transform to a related path transformation.

We begin with some key definitions and classical results. Let
$(B(t),\ t\in[0,1])$ denote standard real-valued Brownian motion. Let
$(\Bbr(t),\ t\in[0,1])$ denote a standard Brownian bridge and
$(\Bex(t),\ t\in [0,1])$ a standard Brownian excursion -- see, for
example, M\"{o}rters and Peres\cite{MortersPeres} or
Billingsley\cite{Billingsley} for the definitions of these
processes. When we wish to make statements or definitions which apply
to all three of $B$, $\Bbr$, and $\Bex$, we  use $(X(t),\ t\in
[0,1])$ to denote a general pick from among these. Finally, we use
`$\stackrel{d}{=}$' to denote equality in distribution.

\begin{Def}\label{def:LT}
 We use $\ell_t(a)$ to denote an a.s.\ jointly continuous version of
 the \emph{(occupation density) local time} of $X$ at level $a$, up to
 time $t$. That is
 \begin{align}
  \ell_t(a) = \lim_{\e\downto 0} \frac{1}{2\e} \int_0^1 \cf\{|X(s)-a| < \e\} ds.\label{eq:LT_defn}
 \end{align}
 The existence of an a.s.\ jointly continuous version is well known,
 and is originally due to Trotter\cite{Trotter58}. We  often
 abbreviate
 $$\ell(a) := \ell_1(a).$$ Let $F(a)$ denote the \emph{cumulative
   distribution function} (or \emph{CDF}) \emph{of occupation
   measure},
 \begin{align}
  F(a) := \int_{-\infty}^a \ell^y dy = \rm{Leb}\{s\in [0,1] : X(s)\leq a\}.\label{eq:CDF_defn}
 \end{align}
 By the continuity of $X$, the function $F$ is strictly increasing in
 between its escape from 0 and arrival at 1. Thus we may define an
 inverse of $F$, the \emph{quantile function of occupation measure},
 \begin{align}
  A(s) := \inf\{a : F(a) > s\}\rm{\ for\ }s\in [0,1),\label{eq:QF_OM_defn}
 \end{align}
 and we extend this function continuously to define $A(1) :=
 \max_{s\in [0,1]}X(s)$.
\end{Def}

Recall that for a walk $w$, the value $Q(w)(j)$ is the sum of
increments from $w$ which appear at the $j$ lowest values in the path
of $w$. Heuristically, at least, the continuous-time analogue to this
is the formula
\begin{align}
 Q(X)(t) = \int_0^1 \cf\{X(s)\leq A(t)\} dX(s).\label{eq:q_integral_rep}
\end{align}
This formula would define $Q(X)(t)$ as the sum of bits of the path of
$X$ which emerge from below a certain threshold -- the exact threshold
below which $X$ spends a total of time $t$. But it is unclear how to
make sense of the integral: it cannot be an It\^{o} integral because
the integrand is not adapted. Perkins\cite[p.\ 107]{Perkins82} allows
us to make sense of this and similar integrals. We quote Tanaka's
formula:
\begin{align}
 \int_0^1 \cf\{X(s)\leq a\} dX(s) = \frac12\ell(a) + (a)_+ - (a-X(1))_+,\label{eq:Tanaka}
\end{align}
where $(c)_+$ denotes $\max(c,0)$. For more on Tanaka's formula see
e.g.\ Karatzas and Shreve\cite[p.\ 205]{KaratzasShreve}. The
particulars of Perkins' result are not important here -- we quote it
only as motivation. He defines
\begin{align}
 \int_0^1 \cf\{X(s)\leq A(t)\} dX(s) &:= \int_{-\infty}^{\infty} \cf\{a \leq A(t)\}dJ(a)\\
 	&= \int_{-\infty}^{\infty} \cf\{F(a) \leq t\}dJ(a),
\end{align}
where $J(a)$ equals the right-hand side of \eqref{eq:Tanaka}, which is
a semi-martingale with respect to a certain naturally arising
filtration. This motivates us in the following definition.

\begin{Def}\label{def:q_cts}
 The \emph{quantile transform of Brownian motion / bridge / excursion} is
 \begin{align}
  Q(X)(t) := \frac12\ell(A(t)) + (A(t))_+ - (A(t)-X(1))_+.\label{eq:q_cts_defn}
 \end{align}
 In the bridge and excursion cases this expression reduces to
 \begin{align}
  Q(X)(t) := \frac12\ell(A(t)).
 \end{align}
\end{Def}

We call upon classic limit results relating Brownian motion and its
local times to their analogues for simple random walk. The work here
falls into the broader scheme of limit results and asymptotics
relating random walk local times to Brownian local times. We rely
heavily on two results of Knight\cite{Knight62,Knight63} in this
area. Much else has been done around local time asymptotics; in
particular, Cs\'aki, Cs\"org\H{o}, F\"oldes, and R\'ev\'esz have
collaborated extensively, as a foursome and as individuals and pairs,
in this area. We mention a small segment of their work:
\cite{Revesz,Revesz81,CsorReve84,CsorReve85_1,CsorReve85_2,CsakCsorFoldReve09}. See
also Bass and Khoshnevisan\cite{BassKhos93_1,BassKhos93_2} and
Szabados and Sz\'ekeley\cite{SzabSzek05}.

\begin{Def}
 For each $n\geq 1$ let $\tau_n(0) := 0$ and
 \begin{align}
  \tau_n(j) := \inf\{t > \tau_n(j-1) : B(t) - B(\tau_n(j-1)) = \pm2^{-n}\}\rm{\; for\ }j\in(0,4^n].
 \end{align}
 We define a walk
 \begin{align}
  S_n(j) &:= 2^nB(\tau_n(j))\rm{\ for\ }j\in [0,4^n]\rm{ and}\\
  \bar S_n(t) &:= 2^{-n}S_n([4^nt])\rm{\ for\ }t\in [0,1].
 \end{align}
 From elementary properties of Brownian motion, $(S_n(j),\ j\geq 0)$
 is a simple random walk. We call the sequence of walks $S_n$ the
 \emph{simple random walks embedded in $B$}. Since we will be dealing
 with the quantile transformed walk $Q(S_n)$, we define a rescaled
 version:
 \begin{align*}
  \Qbar(t) := 2^{-n}Q(S_n)([4^nt]).
 \end{align*}
\end{Def}

Note that $\tau^n_{4^n}$ is the sum of $4^n$ independent,
\emph{Exp}($4^{n}$)-distributed variables. By a Borel-Cantelli
argument, the $\tau_n(4^n)$ converge a.s.\ to 1. So the walks $S_n$
depend upon the behavior of $B$ on an interval converging a.s.\ to
$[0,1]$ as $n$ increases.

The remainder of this section works to prove that, as $n$ increases,
$\Qbar$ almost surely converge uniformly to $Q(B)$.

\begin{Def}
 We define the \emph{(discrete) local time} of $S_n(j)$ at level $x \in \mathbb R$
 \begin{align*}
  L_n(x) &:= \sum_{j=0}^{4^n-1}(1-(x-[x]))\cf\{S_n(j) = [x]\} + (x-[x])\cf\{S_n(j) = [x]+1\}
 \end{align*}
 This is a linearly interpolated version of the standard discrete
 local time. We also require a rescaled version,
 \begin{align*}
  \bar L_n(x) &:= 2^{-n}L_n(2^nx).
 \end{align*}
\end{Def}

Note that for $x\in \mathbb Z$ we get
\begin{align*}
  L_n(x) &= \#\{j\in [0,4^n) : S_n(j) = x\}\rm{\ and}\\
 \bar L_n(2^{-n}x) &= \rm{Leb}\{t\in [0,1]\ :\ \bar S_n(t) = 2^{-n}x\}.
\end{align*}

We note that previous authors have stated convergence results for a
discrete version of Tanaka's formula. See Szabados and
Szekely\cite[p.\ 208-9]{SzabSzek09} and references therein. However,
these results are not applicable in our situation due to the random
time change $A(t)$ that appears in our continuous-time formulae.

We require several limit theorems relating simple random walk and its
local times to Brownian motion, summarized below.

\begin{Thm}
 \begin{align}
  &\bar S_n(\cdot) \to B(\cdot)\rm{\ a.s.\ uniformly}  &\rm{(Knight, 1962\cite{Knight62}).}\label{eq:Knight_walk}\\
  &\hspace{-4pt}\begin{array}{l}
   \min_t\{\bar S_n(t)\}\ \to\ \min_{t\in [0,1]}B_t\rm{\ and}\\
   \max_t\{\bar S_n(t)\}\ \to\ \max_{t\in [0,1]}B_t
  \end{array}	&\rm{(corollary to above).}\label{eq:max_min_cnvgc}\\
  &\bar L_n(\cdot) \to \ell(\cdot)\rm{\ a.s.\ uniformly}  &\rm{(Knight, 1963\cite{Knight63}).}\label{eq:Knight_LT}
 \end{align}
\end{Thm}

Equation \eqref{eq:Knight_walk} is an a.s.\ variant of Donsker's
Theorem, which is discussed in standard textbooks such as
Durrett\cite{Durrett} and Kallenberg\cite{Kallenberg}. Equation
\eqref{eq:max_min_cnvgc} is a corollary to the Knight result: both max
and min are continuous with respect to the uniform convergence
metric. The map from a process to its local time process, on the other
hand, is not continuous with respect to uniform convergence; thus,
equation \eqref{eq:Knight_LT} stands as its own result. An elementary
proof of this latter result, albeit with convergence in probability
rather than a.s., can be found in \cite{Revesz81}, along with a sharp
rate of convergence. Knight\cite{Knight97} gives a sharp rate of
convergence under the $L^2$ norm.

\begin{Def}
 The \emph{cumulative distribution function (CDF) of occupation
   measure} for $S_n$, denoted by $F_n$, is given by
 \begin{align*}
  F_n(y) &:= \int_{-\infty}^y L_n(x)dx\rm{ and}\\
  \bar{F}_n(y) &:= 4^{-n}F_n(2^ny) = \int_{-\infty}^y \bar L_n(x)dx.
 \end{align*}
\end{Def}

Compare these to $F$, the CDF of occupation measure for $B$, defined
in equation \eqref{eq:CDF_defn}. We have restated it to highlight the
parallel to $F_n$. Also note that for integers $k$,
\begin{align}
 F_n(k) &= \sum_{j<k}L_n(j) + \frac12 L_n(k)\label{eq:discrete_LT_CDF}\\
 	&= \#\{i\in [0,4^n) : S_n(i) < k\} + \frac{1}{2}\#\{i\in [0,4^n) : S_n(i) = k\}.\notag
\end{align}

Equations \eqref{eq:Knight_LT} and \eqref{eq:max_min_cnvgc} have the
following easy consequence.
\begin{Cor}\label{cor:CDF_OM_cnvgc}
 As $n$ increases the $\bar F_n$ a.s.\ converge uniformly to $F$.
\end{Cor}

Because Brownian motion is continuous and simple random walk cannot
skip levels, the CDFs $F$ and $F_n$ are strictly increasing between
the times where they leave 0 reach their maxima, 1 or $4^n$
respectively. This admits the following definitions.

\begin{Def}
 We define the \emph{quantile functions of occupation measure}
 \begin{align*}
  A_n(t) &:= F_n^{-1}(t)\rm{ for }t\in (0,4^n)\rm{, and}\\
  \bar A_n(t) &:= \bar F_n^{-1}(t)\rm{ for }t\in (0,1),
 \end{align*}
 and we extend these continuously to define $A_n(0)$, $\bar A_n(0)$, $A_n(4^n)$ and $\bar A_n(1)$.
\end{Def}

Compare these to $A$ defined in equation \eqref{eq:QF_OM_defn} in the introduction.

\begin{Lem}\label{lem:quantile_cnvgc}
 As $n$ increases the $\bar A_n$ a.s.\ converge uniformly to $A$.
\end{Lem}

\begin{proof}
 In passing a convergence result from a function to its inverse it is
 convenient to appeal to the Skorokhod metric. For continuous
 functions, uniform convergence on a compact interval $I \subset
 \mathbb R$ is equivalent to convergence under the Skorohod metric
 (see \cite{Billingsley}). Let $i$ denote the identity map on $I$, let
 $||\cdot||$ denote the uniform convergence metric, and let $\Lambda$
 denote the set of all increasing, continuous bijections on $I$. The
 Skorokhod metric may be defined as follows:
 \begin{align}
  \sigma(f,g) := \inf_{\lambda \in \Lambda}\max\{||i-\lambda||,\ ||f - g\circ\lambda||\}.
 \end{align}
 Thus, it suffices to prove a.s.\ convergence under $\sigma$.

 Fix $\e > 0$. By the continuity of $A$, there is a.s.\ some
 $0<\delta<\e$ sufficiently small so that
 $$A(\delta) - \min_{[0,1]}B(t) < \e\rm{ and }\max_{[0,1]} B(t) -
 A(1-\delta) < \e.$$ And by Equation \eqref{eq:max_min_cnvgc} and
 Corollary \ref{cor:CDF_OM_cnvgc} there is a.s.\ some $n$ so that, for
 all $m\geq n$,
 \begin{align*}
  \min_{t\in [0,1]} \bar S_m(t) &< A(\delta);\\
  \max_{t\in [0,1]} \bar S_m(t) &> A(1-\delta)\rm{; and}\\
  \sup_y |\bar F_m(y) - F(y)| &< \e.
 \end{align*}

 We  show that $\sigma(\bar A_n,A) < 3\e$.

 We seek a time change $\lambda: [0,1]\to [0,1]$ which is close to the
 identity and for which $\bar A_n\circ\lambda$ is close to
 $A$. Ideally, we would like to define $\lambda = \bar F_n\cdot A$ so
 as to get $\bar A_n\circ\lambda = A$ exactly. But there is a problem
 with this choice: because $\bar S_n$ and $B$ may not have the exact
 same max and min, $\bar F_n\circ A$ may not be a bijection on
 $[0,1]$. We turn this map into a bijection by manipulating its values
 near 0 and 1.

 We define the random time change on $[0,1]$
 \begin{align}
  \lambda(t) := \left\{\begin{array}{ll}
 			\frac{t}{\delta}\bar F_n(A(\delta))	& \rm{for }0\leq t < \delta\\
 			\bar F_n(A(t))	& \rm{for }\delta\leq t \leq 1-\delta\\
			1 + \frac{1-t}{\delta}(\bar F_n(A(1-\delta)) - 1)	& \rm{for }1-\delta<t\leq 1.
                          \end{array}\right.
 \end{align}

 By our choice of $n$ we get
 \begin{align*}
  \bar F_n(A(\delta)) > 0\rm{\; and\; }\bar F_n(A(1-\delta)) < 1.
 \end{align*}
 Thus $\lambda$ is a bijection.

 We now show that it is uniformly close to the identity. Since $t =
 F(A(t))$, our conditions on $n$ give us
 \begin{align*}
  ||\lambda(t) - t||_{t\in[\delta,1-\delta]} \leq ||bar F_n(A(t)) - F(A(t))|| < \e.
 \end{align*}
 For $t$ near 0
 \begin{align*}
  ||\lambda(t) - t||_{t<\delta} \leq |\lambda(\delta) - F(A(\delta))| < \e,
 \end{align*}
 and likewise for $t > 1 - \delta$.

 Next we consider the difference between $A$ and $\bar
 A_n\circ\lambda$. These are equal on $[\delta,1-\delta]$. For $t <
 \delta$ we get
 \begin{align*}
  A(t)\in[(\min_t B_t),\ A(\delta)]\rm{ and }\bar A_n\circ\lambda(t)\in [(\min_t \bar S_n(t)),\ A(\delta)].
 \end{align*}
 By our choices of $n$ and $\delta$, the lower bounds on these
 intervals both lie within $2\e$ of $\delta$. A similar argument works
 for $t > 1-\delta$. Thus $A(t)$ lies within $2\e$ of $\bar
 A_n\circ\lambda(t)$.

 We conclude that $\sigma(\bar A_m,a) < 3\e$ for $m\geq n$.
\end{proof}

For our purpose, the important consequence of the preceding lemma is
the following.

\begin{Cor}\label{cor:time_changed_cnvgc}
 As $n$ increases the $\bar L_n \circ \bar A_n$ a.s.\ converge
 uniformly to $\ell_1 \circ A$.
\end{Cor}

General results for convergence of randomly time-changed random
processes can be found in Billingsley\cite{Billingsley}, but in the
present case the proof of Corollary \ref{cor:time_changed_cnvgc} from
equation \eqref{eq:Knight_LT} and Lemma \ref{lem:quantile_cnvgc} is an
elementary exercise in analysis, thanks to the a.s.\ uniform
continuity of $\ell$.

We now make use of the up- and down-crossing counts described
in Definition \ref{def:crossing_counts}, and of the saw teeth in
Definition \ref{def:saw_tooth}. For our present purpose it is
convenient to re-index these sequences.

\begin{Def}
 Let $m_n = \min_{j<4^n}S_n(j)$. For each $i\geq m_n$ we define
 $u^n_i$ to be the number of up-steps of $S_n$ which go from the value
 $i$ to $i+1$. Likewise, let $d^n_i$ denote the number of down-steps
 of $S_n$ from value $i$ to $i-1$. Finally, let
 \begin{align}
  t^n_i = \sum_{j<i}(u^n_j+d^n_j).
 \end{align}
 We call these quantities \emph{up-} and \emph{down-crossing counts}
 and \emph{saw teeth}.
\end{Def}

Note that the strict inequality in the bound on $j$ in the definition
of $m_n$ is necessary.

Comparing the sequence $(u^{S_n}_i)$ in Definition
\ref{def:crossing_counts} with the sequence $(u^n_i)$, we have
$$u^n_i = u ^{S_n}_{i+m_n}.$$
Comparing the sequence $(t^{S_n}_i)$ defined in Definition
\ref{def:saw_tooth} with the sequence $t^n_i$, we have
\begin{align*}
  t^n_i = t^{S_n}_{i+m_n}.
\end{align*}

Note that
\begin{align}
 L_n(k) = u^n_k + d^n_k = t^n_{k+1} - t^n_k.\label{eq:LT_tooth_gaps}
\end{align}

At saw tooth times, the quantile transform $Q(S_n)$ is uniformly well
approximated by a formula based on discrete local time.

\begin{Lem}\label{lem:saw_tooth_cnvgc}
 Let $A^n_k$ denote $A_n(t^n_k)$. As $n$ increases the following
 quantities a.s.\ vanish uniformly in $k$:
 \begin{enumerate}\setlength{\itemsep}{0pt}
  \item $2^{-n}|L_n(k) - 2u^n_k|$,
  \item $2^{-n}|F_n(k) - t^n_k|$,
  \item $|F(2^{-n}k) - 4^{-n}t^n_k|$,
  \item $2^{-n}|A^n_k - k|$, and
  \item $\displaystyle 2^{-n}\left|Q(S_n)(t^n_k) - \left(\frac{1}{2}L_n(A^n_k) + (A^n_k)_+ - (A^n_k - S_n(4^n))_+\right)\right|.$ \label{item:cnvgc_at_saw_teeth}
 \end{enumerate}
\end{Lem}

\begin{proof}
 The convergence of (ii) follows from that of (i) by equation
 \eqref{eq:discrete_LT_CDF}, which gives us
 \begin{align}
  F_n(k) = t^k_{n} + \left(\frac12 L_n(k) - u^n_k\right)
 \end{align}
 for integers $k$; (iii) then follows by Corollary
 \ref{cor:CDF_OM_cnvgc}. The convergence of (iv) follows from that of
 (ii) by Lemma \ref{lem:quantile_cnvgc} and the uniform continuity of
 $a$. And finally, (v) then follows from the others by the discrete
 Tanaka formula, equation \eqref{eq:saw}. Note that by re-indexing, we
 have replaced the $\cS$ and $\cT$ from that formula, which are the
 start and terminal levels, with $0$ and $S_n(4^n)$ respectively,
 which are the start and terminal values of $S_n$. Thus, it suffices
 to prove the convergence of (i).

 If we condition on $L_n(k)$ then $u^n_k$ is distributed as
 \emph{Binomial}$(L_n(k),\frac12)$. Our intuition going forward is
 this: if $L_n(k)$ is large then $(L_n(k) - 2u^n_k)/\sqrt{L_n(k)}$
 approximates a standard Gaussian distribution. Throughout the
 remainder of the proof, let $binom(n)$ denote a
 \emph{Binomial}$(n,\frac12)$ variable on a separate probability
 space.
 Fix $\e>0$ and let
 \begin{align*}
  C_1 = 1 + \max_t |B(t)|\rm{\; and\;}C_2 = 1 + \max_x \ell(x).
 \end{align*}
 Let $M$ be sufficiently large so that for all $n\geq M$,
 $$\Pr\left\{ |2^nC_2 - 2binom(2^nC_2)| > 2^n\e\right\} <
 \sqrt{2/\pi}\exp (-2^{n-1}\e^2/C_2).$$ Such an $M$ must exist by the
 central limit theorem and well-known bounds on the tails of the
 normal distribution. Let $N \geq M$ be sufficiently large so that for
 all $n\geq N$,
 $$\max_t|S_n(t)| < 2^nC_1\rm{ and }\max_{x}L_n(x) < 2^nC_2.$$
 Equations \eqref{eq:Knight_LT} \eqref{eq:max_min_cnvgc} indicate that
 $N$ is a.s.\ finite.

 We now apply the Borel-Cantelli Lemma.
 \begin{align*}
  &\sum_{n>M}\sum_k\Pr\left\{|L_n(k) - 2u^n_k| > 2^n\e;\ n > N\right\}\\
  	&\leq \sum_{n>M}2^{n+1} C_1\max_k\Pr\left\{|L_n(k) - 2u^n_k| > 2^n\e;\ n > N\right\}\\
  	&< \sum_{n>M}2C_1 e^n\max_{y\leq 2^nC_2}\Pr\{|y - 2binom(y)| > 2^n\e\}\\
  	&< \sum_{n>M}C_1\sqrt{\frac{8}{\pi}}\exp (n-(2^{n-1}\e^2/C_2)) < \infty.
 \end{align*}
 The claimed convergence follows by Borel-Cantelli.
\end{proof}

Our proof implicitly appeals to the branching process view of Dyck
paths. This perspective may be originally attributable to
Harris\cite{Harris52} and was implicit in the Knight papers
\cite{Knight62,Knight63} cited earlier in this section. See also
\cite{Pitman99} and the references therein.

In order to prove Theorem \ref{thm:gen_Jeulin}, we must extend the
convergence of (v) in the previous lemma to times between the saw
teeth. The convergence of (iii) leads to a helpful corollary.

\begin{Cor}\label{cor:uniform_saw_teeth}
  The sequence $\min_k |t - 4^{-n}t^n_k|$ a.s.\ converges to 0
  uniformly for $t\in [0,1]$.
\end{Cor}

\begin{proof}
 Since $\min_k t^n_k = 0$ and $\max_k t^n_k = 4^n$, it suffices to
 prove that $4^{-n}\sup_k (t^n_k - t^n_{k-1})$ a.s.\ converges to
 0. This follows from: the uniform continuity of $F$, the uniform
 convergence of the $\bar F_n$ to $F$ asserted in Corollary
 \ref{cor:CDF_OM_cnvgc}, and the uniform vanishing of $|\bar F_n(k) -
 2^{-n}t^n_k|$ asserted in Lemma \ref{lem:saw_tooth_cnvgc}.
\end{proof}

We now prove a weak version of Theorem \ref{thm:q_cts_limit} before
demonstrating the full result.

\begin{Lem}\label{lem:q_lim_teeth}
 Let $Z_n$ be the process which equals $Q(S_n)$ at the saw teeth and
 is linearly interpolated in between, and let $\bar Z_n$ be the
 obvious rescaling. As $n$ increases, $\bar Z_n$ a.s.\ converges
 uniformly to $Q(B)$.
\end{Lem}

\begin{proof}
  Let
 \begin{align}
  \bar X_n(t) := \frac{1}{2}\bar L_n(\bar A_n(t)) + (\bar A_n(t))_+ -
  (\bar A_n(t) - \bar S_n(1))_+,\label{eq:discr_LT_q}
 \end{align}
 and let $\bar Y_n$ denote the process which equals $\bar X_n$ at the
 (rescaled) saw teeth $4^{-n}t^n_k$ and is linearly interpolated
 between these times. We prove the lemma by showing that the following
 differences of processes go to 0 uniformly as $n$ increases: (i)
 $\bar X_n - Q(B)$, (ii) $\bar Y_n - \bar X_n$, and (iii) $\bar Z_n -
 \bar Y_n$.

 The uniform vanishing of (i) follows from equations
 \eqref{eq:Knight_walk} and \eqref{eq:Knight_LT}, Lemma
 \ref{lem:quantile_cnvgc}, and Corollary
 \ref{cor:time_changed_cnvgc}. That of (iii) is equivalent to item (v)
 in Lemma \ref{lem:saw_tooth_cnvgc}. Finally, each of the three terms
 on the right in equation \eqref{eq:discr_LT_q} converge uniformly to
 uniformly continuous processes, so by Corollary
 \ref{cor:uniform_saw_teeth}, $(\bar Y_n - \bar X_n)$ a.s.\ vanishes
 uniformly as well.
\end{proof}

Before the technical work of extending this lemma to a full proof of
Theorem \ref{thm:q_cts_limit} we mention a useful bound.  For a simple
random walk bridge $(D(j),j\in [0,2n])$,
\begin{align}
 \Pr (\max_{j\in[0,2n]}|D(j)| \geq c\sqrt{2n}) \leq 2e^{-c^2}.\label{eq:max_modulus_SRWB}
\end{align}
This formula may be obtained via the reflection principle and some
approximation of binomial coefficients; we leave the details to the
reader. The Brownian analogue to this bound appears in
Billingsley\cite[p.\ 85]{Billingsley}:
\begin{align}
 \Pr (\sup_{t\in [0,1]}|B^{\rm{br}}(t)| > c) \leq 2e^{-2c^2}.\label{eq:max_modulus_BM}
\end{align}
For our purposes the `$2$' in the exponent above is unnecessary, so
we've sacrificed it to keep our discrete-time inequality
\eqref{eq:max_modulus_SRWB}.

\begin{Lem}
 Fix $\e,\d > 0$. Let $(\lambda^n_k)_{n,k\geq 0}$ be a family of random non-negative integers and $(W^n_k)_{n,k\geq 0}$ a family of walks, each having length $\lambda^n_k$ and exchangeable increments of $\pm1$. Suppose that the $W_{nk}$ are mutually independent conditional on $\{\lambda^n_k,W^n_k(\lambda^n_k)\}_{n,k\geq 0}$. And suppose further that there is some a.s.\ finite $N$ such that, for $n\geq N$:
 $$\sup_k W_{nk}(\lambda^n_k) \leq 3^n\d\rm{, and }\sup\{k : \lambda^n_k > 0\} < n2^{n+1}.$$
 Then the largest $n$ for which
 $$\sup_{j\in [0,\lambda^n_k],\ k\geq 0}|W^n_k(j) - \frac{j}{\lambda^n_k}W^n_k(\lambda^n_k)| > 2^n\e$$
 is a.s.\ finite.
\end{Lem}

\begin{proof}
 We prove this with a coupling argument. First, we observe that
 \[\sup_j\left|W^n_k(j) - \frac{j}{\lambda^n_k}W^n_k(\lambda^n_k)\right| < \sup_j\sup\{|W^n_k(j)|,\ |W^n_k(j) - W^n_k(\lambda^n_k)|\}.\]
 Next, we introduce a family of random walks $D^n_k$ which, conditional on $(\lambda^n_k)_{n,k\geq 0}$, are independent of each other and of the $W^n_k$. Let $D^n_k$ be a simple random walk bridge to 0 in the case where $\lambda^n_k$ is even, or to 1 in the case where $\lambda^n_k$ is odd.

 Now let $W^n_k$
\end{proof}

We now arrive at our main result.

\begin{Thm}\label{thm:q_cts_limit}
 As $n$ increases, $\Qbar$ a.s.\ converges uniformly to $Q(B)$.
\end{Thm}

\begin{proof}
 Let $Z_n$ and $\bar Z_n$ be as in Lemma \ref{lem:q_lim_teeth}. After
 that lemma it suffices to prove that $(\Qbar - \bar Z_n)$ vanishes
 uniformly as $n$ increases. By definition, this difference equals 0
 at the saw teeth. Moreover, we deduce from Theorems
 \ref{thm:array_to_walk} and \ref{thm:saw_to_walk} that conditional on
 $Z_n$, the walk $Q(S_n)$ is a simple random walk conditioned to equal
 $Z_n$ at the saw teeth $t^n_k$ and with some constraints, coming from
 the Bookends property, on its $(t^n_k)^{\rm{th}}$ steps.

 We must bound the fluctuations of $Q(S_n)$ in between the saw
 teeth. Heuristic arguments suggest that these ought to have size on
 the order of $2^{n/2}$; we need only show that they grow uniformly
 slower than $2^n$. We  prove this via a Borel-Cantelli
 argument. There are many ways to bound the relevant probabilities of
 ``bad behavior;'' we proceed with a coupling argument.

 For each $(n,k)$ for which $t^n_k$ is defined -- i.e.\ with $k\in
 [\min S_n,\ \max S_n]$ -- we define several processes and stopping
 times. These objects  appear illustrated together in figure
 \ref{fig:coupling}. First, for $j\in [0,L_n(k)-1]$ we define
 \begin{align*}
  \hat W^n_k(j) &:= Q(S_n)(t^n_k + j) - Q(S_n)(t^n_k)\rm{ and}\\
  \check W^n_k(j) &:= Q(S_n)(t^n_k + j) - Q(S_n)(t^n_{k+1} - 1).
 \end{align*}
 Recall from equation \eqref{eq:LT_tooth_gaps} that $L_n(k)$ is the
 difference between consecutive saw teeth. We only define these walks
 up to time $L^n_k - 1$ so as to sidestep issues around constrained
 final increments and the bookends property. Observe that
 \begin{align*}
  \max_{j\in[t^n_k,t^n_{k+1}]}|Q(S_n)(j) - Z_n(j)| \leq 1 + \max_{j\in [0,L_n(k)-1]}\{ |\hat W^n_k(j)|,|\check W^n_k(j)|\},
 \end{align*}
 so it suffices to bound the fluctuations of the $\hat W$ and
 $\check W$.

 We further define
 \begin{align*}
  \Delta^n_k := Q(S_n)(t^n_{k+1}-1) - Q(S_n)(t^n_k)
 \end{align*}
 Observe that
 \begin{align}\label{eq:Wnk_boundary}
  \begin{array}{r@{\;\;\rm{and}\;\;}l}
     \hat W^n_k(0) = 0             & \hat W^n_k(L_n(k) - 1) = \Delta^n_k\rm{,\ whereas}\\
     \check W^n_k(0) = -\Delta^n_k & \check W^n_k(L_n(k) - 1) = 0.
  \end{array}
 \end{align}

 If $L_n(k)$ is an odd number then we may define a simple random walk
 bridge $D^n_k$ that has random length $L_n(k)-1$ but is otherwise
 independent of $S_n$ (we enlarge our probability space as necessary
 to accommodate these processes). In the next paragraph we deal with
 the case where $L_n(k)$ is even. Let
 \begin{align*}
  \hat T^n_k &:= \min\{ j : D^n_k(j) + \Delta^n_k = \hat W^n_k(j)\}\rm{\ and}\\
  \check T^n_k &:= \max\{ j : D^n_k(j) - \Delta^n_k = \check W^n_k(j)\}.
 \end{align*}
 These stopping times must be finite, thanks to the values of $\hat W$
 and $\check W$ observed in \eqref{eq:Wnk_boundary}. Finally we define
 the coupled walks.
 \begin{align}
  \hat D^n_k(j) &= \left\{\begin{array}{ll}
    D^n_k(j) + \Delta^n_k    &\rm{for\ }j \in [0,\hat T^n_k]\\
    \hat W^n_k(j)    & \rm{for\ }j \in (\hat T^n_k,L_n(k)-1].
  \end{array}\right.\\
  \check D^n_k(j) &= \left\{\begin{array}{ll}
    \check W^n_k(j)    & \rm{for\ }j \in [0,\check T^n_k]\\
    D^n_k(j) - \Delta^n_k    &\rm{for\ }j \in (\check T^n_k,L_n(k)-1].
  \end{array}\right.
 \end{align}
 Conditional on $L_n(k)$, the $\hat D^n_k$ and $\check D^n_k$ remain
 simple random walk bridges, albeit vertically translated. These are
 illustrated in Figure \ref{fig:coupling}.

 In the case where $L_n(k)$ is even rather than odd, we modify the
 above definitions by making $D^n_k$ a bridge to $-1$ if $\Delta^n_k >
 0$ (or $1$ respectively if $\Delta^n_k < 0$) instead of 0 and
 including appropriate `$+1$'s (respectively `$-1$'s) into the
 definitions of $\hat T^n_k$ and $\hat D^n_k$ so that the final value
 of $D^n_k + \Delta^n_k + 1$ (resp.\ $-1$) aligns with that of $\hat
 W^n_k$.

 \begin{figure}\centering
  \input{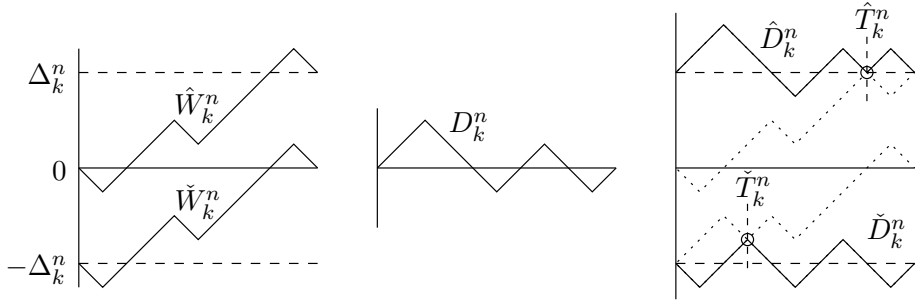}
  \caption{Objects from the coupling argument.}\label{fig:coupling}
 \end{figure}

 Fix $\e>0$. We may bound the extrema of $\hat W^n_k$ and $\check
 W^n_k$ by bounding the extrema of $\hat D^n_k$ and $\check D^n_k$. In
 particular, we have the following event inclusions.
 \begin{align}
  &\{\max_j\{|\hat W^n_k(j)|,\ |\check W^n_k(j)|\} \geq 2^{n+1}\e\}\notag\\
    &\subseteq \{\max_j \{|\hat D^n_k(j)|,\ |\check D^n_k(j)|\} \geq 2^{n+1}\e\}\notag\\
    &\subseteq \{|\Delta^n_k| + 1 \geq 2^{n}\e\} \cup \{\max_j |D^n_k(j)| \geq 2^{n}\e\}.\label{eq:large_modulus_event}
 \end{align}
 First we  use previous results from this section to prove that
 a.s.\ only finitely many of the $\Delta^n_k$ are large. Then we make
 a Borel-Cantelli argument to do the same for the $\max_j |D^n_k(j)|$.

 By the continuity of $Q(B)$, there is a.s.\ some $\d\in (0,\e^2)$
 sufficiently small so that
 $$\max_{|t-s|<\d} |Q(B)(t) - Q(B)(s)| < \e.$$ And there is a.s.\ some
 $N$ sufficiently large so that for $n\geq N$:
 \begin{align*}
  \sup_j |S_n(j)| &< n2^n,\\
  \max_k L_n(k) &< 3^n\delta\rm{, and}\\
  \sup_t |\bar Z_n(t) - Q(B)(t)| &< \e.
 \end{align*}
 The first two of these bounds follow from the continuity of $\ell$
 and equations \eqref{eq:Knight_walk} and \eqref{eq:Knight_LT}; the
 third follows from Lemma \ref{lem:q_lim_teeth}. The second and third
 of these imply that for $n\geq N$,
 \begin{align*}
  |\Delta^n_k| &\leq |Z_n(t^n_{k+1}) - 2^nQ(B)(4^{-n}t^n_{k+1})| + 2^n|Q(B)(4^{-n}t^n_{k+1}) - Q(B)(4^{-n}t^n_k)|\\
  	&\eqspace + |2^nQ(B)(4^{-n}t^n_k) - Z_n(t^n_k)|\\
  	&\leq 3\cdot 2^n\e.
 \end{align*}
 So, folding constants into $\e$, there is a.s.\ some largest $n$ for
 which any of the $|\Delta^n_k|$ exceed $2^n\e$.

 We proceed to our Borel-Cantelli argument to bound fluctuations in the $D^n_k$.
 \begin{align*}
  &\sum_{n}\sum_k \Pr \{\max_{j}|D^n_k(j)| > 2^n\e;\ n>N\}\\
  	&\leq \sum_{n} 2^{n+1} n\max_{|k|< 2^{n}n} \Pr \{\max_j|D^n_k(j)| > 2^n\e;\ n>N\}\\
  	&\leq \sum_{n} 2^{n+1} n\max_{l\leq [3^n\d]} \Pr\{\max_j|D^n_0(j)| > 2^n\e\ |\ L_n(0) = l\}\\
  	&\leq \sum_{n} 2^{n+2} n e^{-(\frac43)^n} < \infty.
 \end{align*}
 The last line above follows from \eqref{eq:max_modulus_SRWB}. We
 conclude from the Borel-Cantelli Lemma that a.s.\ only finitely many
 of the $D^n_k$ exceed $2^n\e$ in maximum modulus. So by the event
 inequality \eqref{eq:large_modulus_event}, a.s.\ only finitely many
 of the $W^n_k$ exceed $2^{n+1}\e$ in maximum modulus.
\end{proof}

Our main result in the continuous setting, Theorem
\ref{thm:gen_Jeulin} below, now emerges as a corollary.

\begin{Def}\label{def:Vervaat_cts}
 Let $\tau_m$ denote the time of the (first) arrival of $(X(t),\ t\in
 [0,1])$ at its minimum. The \emph{Vervaat transform} maps $X$ to the
 process $V(X)$ given by
 \begin{align}
  V(X)(t) := \left\{\begin{array}{ll}
          X(\tau_m + t) - X(\tau_m)	& \rm{for\ }t\in [0, 1-\tau_m)\\
          X(\tau_m + t - 1) + X(1) - X(\tau_m)	& \rm{for\ }t\in [1-\tau_m, 1].
         \end{array}\right.
 \end{align}
\end{Def}

This transform should be thought of as partitioning the increments of
$X$ into two segments, prior and subsequent to $\tau_m$, and swapping
the order of these segments.

\begin{Thm}\label{thm:Vervaat}
  For $U$ an independent \emph{Uniform}$[0,1]$ random variable, we have
  \begin{align}
  (V(\Bbr)(t),\ t\in [0,1]) &\stackrel{d}{=} (\Bex(t),\ t\in [0,1]). &\rm{(Vervaat, 1979\cite{Vervaat79})}\\
  (\tau_m,\ (V(\Bbr)(t),\ t\in [0,1])) &\stackrel{d}{=} (U,\ (\Bex(t),\ t\in [0,1])). &\rm{(Biane, 1986\cite{Biane86})}
 \end{align}
\end{Thm}

We demonstrated in section \ref{sec:Vervaat} that for simple random
walks, the discrete-time analogue of the Vervaat transform of the walk
has the same distribution as the quantile transform. Now we have shown
that $Q(B)$ arises as an a.s.\ limit of the quantile transforms of
certain simple random walks.

\begin{Thm}\label{thm:gen_Jeulin}
 We have $(Q(B), B(1)) \stackrel{d}{=} (V(B), B(1))$.
\end{Thm}

\begin{proof}
 Let $\overline{V(S_n)}(t) := 2^{-n}V(S_n)([4^nt])$. Vervaat proved
 that $\overline{V(S_n)}$ converges in distribution to $V(B)$. By
 Corollary \ref{cor:Vervaat_count} we have $\Qbar \stackrel{d}{=}
 \overline{V(S_n)}$, and by Theorem \ref{thm:q_cts_limit} the $\Qbar$
 converge in distribution to $Q(B)$. Thus $Q(B)\stackrel{d}{=} V(B)$
 as desired.
\end{proof}

We may use properties of Brownian bridge to give a unique family of
distributions for $Q(B)$ and $V(B)$ conditional on $B(1) = a$ which is
weakly continuous in $a$. In the case where $B(1) = 0$, Theorem
\ref{thm:gen_Jeulin} specializes to the following.

\begin{Thm}[Jeulin, 1985\cite{Jeulin85}]\label{thm:Jeulin}
 If $\ell$ and $A$ denote the local time and the quantile function of
 occupation measure, respectively, of a Brownian bridge or excursion,
 then
 \begin{align}
  \left(\frac 12\ell(A(t)),\ t\in [0,1]\right) \stackrel{d}{=} (\Bex(t),\ t\in [0,1]).
 \end{align}
\end{Thm}

This assertion for Brownian excursions appeared in Jeulin's monograph
\cite[p.\ 264]{Jeulin85} but without a clear, explicit proof; a proof
appears in \cite[p.\ 49]{BianYor87}.

\section{Further connections}
\label{sec:background}

Similar transformations have been widely studied in the
literature. For example, let $x_1,x_2,\cdots$ be a sequence of real
numbers, and define $S(0) = 0$ and
\begin{align}
 S(n) := \sum_{j=1}^n x_j.\label{eq:walk_from_increments}
\end{align}
So the $x_i$ are the increments of the process $S$. Fix some level $l
\geq 0$. We define $S^-(n)$ (and respectively $S^+(n)$) to be the sum
of the first $n$ increments of $S$ which originate at or below
(resp.\ strictly above) the value $l$. That is, an increment $x_i$ of
$S$ is an increment of $S^-$ only if $S(i-1) \leq l$. This is
illustrated in Figure \ref{fig:BCY_transform}; in that example, the
increments $x_3,\ x_7,\ x_8$ and $x_9$ contribute to $S^+(4)$. For the
sake of brevity we omit a more formal definition, which may be found
in \cite{BertChauYor97}. We  call the map $S \mapsto S^-$ the
\emph{BCY transform} (with parameter $l$).

\begin{figure}[htb]\centering
 \input{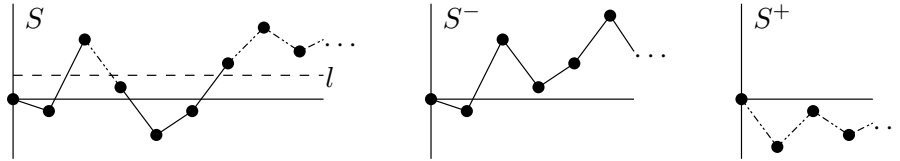}
 \caption{The BCY transform.}\label{fig:BCY_transform}
\end{figure}

The BCY transform resembles the quantile transform in that it sums
increments below some level. But whereas the quantile transform may
only be applied to a walk which has finite length or is upwardly
transient, the BCY transform applies equally well to walk.

There are two big differences between the BCY and quantile
transforms. Firstly, in the case of the BCY transform, the process
$S^-$ comprises all those increments which appear in $S$ below some
previously fixed level $l$; whereas in the case of the quantile
transform, $Q(S)(j)$ comprises (roughly) those increments which appear
in $S$ below a variable level which increases with $j$. Secondly, the
increments of $S^-$ appear in the same order in which they appeared in
$S$, whereas the increments of $Q(S)$ appear in order of the value at
which they appear in $S$.

If we suppose that the $x_i$ are i.i.d.\ random variables then by the
strong Markov property, $S^-$ has the same distribution as $S$
\cite[Lemma 2]{BertChauYor97}. But this is not the case for $Q(S)$;
Theorem \ref{thm:q_bijection} indicates that for $S$ a simple random
walk, $Q(S)$  tend to rise at early times and fall later.

As a further example, the path transformation studied by
Chaumont\cite{Chaumont99} resembles the concatenation of $S^-$
followed by $S^+$, but with some delicate changes. To define it, we
require different notation than that introduced earlier. Recall that
the final increment of the walk $w$ has no bearing on $\Qp_w$. We
require a version of the permutation which does account for this
increment. We draw from the notation of Port\cite{Port63} and
Chaumont\cite{Chaumont99}; this notation is used only in this section
and nowhere else in the paper.

\begin{Def}
 Let the increment sequence $(x_i)_{i=1}^{\infty}$ and the process $S$
 be as above. Let $(S^n(j),\ j\in [0,n])$ denote the restriction of
 $S$ to its $n$ initial increments. For $k\in [0,n]$ we define
 $M^S_{nk}$ and $L^S_{nk}$ so that
 $$(M^S_{n0},L^S_{n0});\ (M^S_{n1},L^S_{n1});\ \cdots;\ (M^S_{nn},L^S_{nn})$$
 is the increasing lexicographic reordering of the sequence
 $$(S(0),0);\ (S(1),1);\ \cdots;\ (S(n),n)$$
 We call the permutation
 $$(0,1,\cdots,\ n) \mapsto (L^S_{n0},\ \cdots,\ L^S_{nn})$$ the
 \emph{quantile permutation of vertices} of $S^n$ (whereas $\Qp_{S^n}$
 might be thought of as a quantile permutation \emph{of
   increments}). We define
 \begin{align*}
  R^S_{nk} &:= \#\{i\leq L^S_{nk}\ :\ S(i)\leq M^S_{nk}\}.
 \end{align*}
\end{Def}
We suppress the superscript when it is clear from context which
process is being discussed.

Both the BCY and Chaumont transforms are motivated by the following
theorem.

\begin{Thm}[Wendel, 1960\cite{Wendel60}; Port, 1963\cite{Port63}; Chaumont, 1999\cite{Chaumont99}]\label{thm:WendelPort}
 Suppose that $x_1,\ \cdots,\ x_n$ are exchangeable real-valued random
 variables, and let $S$ denote the process with these increments. Fix
 $k\in [0,n]$ and let $S'$ denote the process
 $$S'(j) = S(k+j) - S(k)\ \rm{for\ }j\in [0,n-k].$$
 Then
 \begin{align}
  \left(\begin{array}{c}
   S(n)\\
   M^S_{nk}\\
   L^S_{nk}\\
   R^S_{nk}
  \end{array}\right) \stackrel{d}{=}
  \left(\begin{array}{c}
   S(k) + S'(n-k)\\
   M^S_{kk} + M^{S'}_{n-k,0}\\
   L^S_{kk} + L^{S'}_{n-k,0}\\
   L^S_{kk}
  \end{array}\right).\label{eq:WendelPort}
 \end{align}
\end{Thm}

The identity in the first two coordinates in equation
\eqref{eq:WendelPort} is due to Wendel; Port made the (satisfying)
extension of the result to the third coordinate. For more discussion
of related results such as Sparre Andersen's
Theorem\cite{Andersen53_1,Andersen53_2} and Spitzer's Combinatorial
Lemma\cite{Spitzer55}, see Port\cite{Port63}. Port's paper also gives,
on page 140, a combinatorial formula for the probability distribution
of $\Qp_S(j)$ given the distributions of the increments of $S$.

Chaumont made the suggestive extension of \eqref{eq:WendelPort} to the
fourth coordinate and presented the first path-transformation-based
proof Port's result. Let the $x_i$ and $S$ be as in Theorem
\ref{thm:WendelPort} and fix some $k\in [0,n]$. Chaumont's
transformation works by partitioning the increments of $S$ into four
blocks.
\begin{align*}
 I_1 &:= \{i\in [1,L_{nk}]\ :\ S(i-1)\leq M_{nk}\},\\
 I_2 &:= \{i\in (L_{nk},n]\ :\ S(i) < M_{nk}\},\\
 I_3 &:= \{i\in [1,L_{nk}]\ :\ S(i-1) > M_{nk}\},\rm{\ and}\\
 I_4 &:= \{i\in (L_{nk},n]\ :\ S(i)\geq M_{nk}\}
\end{align*}
The \emph{Chaumont transform} sends $S$ to the process $\tilde S$
whose increments are the $x_i$ with $i\in I_1$, followed by those with
$i\in I_2$, then $I_3$, and finally $I_4$, with the increments within
each block arranged in order of increasing index. Details may be found
in \cite[p.\ 3-4]{Chaumont99}. This transformation is illustrated in
Figure \ref{fig:ChaumontXform}, in which increments belonging to $I_1$
and $I_2$ are shown as solid lines, whereas those belonging to $I_3$
and $I_4$ are shown as dotted.

\begin{figure}[htb]\centering
 \input{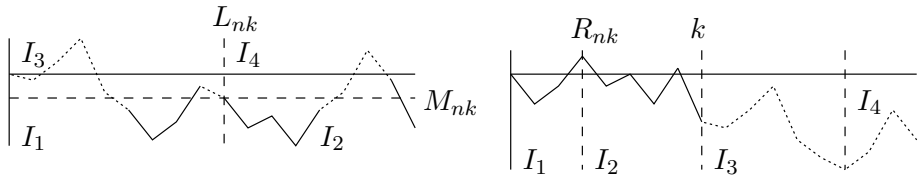}
 \caption{On the left a process $S$ and on the right its Chaumont
   transform.}
 \label{fig:ChaumontXform}
\end{figure}

If $S$ has exchangeable random increments then $S$ and $\tilde S$ have
the same distribution; as with the BCY transform, this presents a
marked difference from the quantile transform. Chaumont demonstrates
that if we substitute $\tilde S$ for $S$ on the right-hand side of
equation \eqref{eq:WendelPort} then we get identical equality, rather
than identity in law.

Theorem \ref{thm:WendelPort} admits various continuous-time
versions. Before stating some of these, we state a loose
continuous-time analogue to the quantile permutation, due to
Chaumont\cite{Chaumont00}.

\begin{Def}
 For $(X(t),\ t\in [0,1])$ a continuous, real-valued stochastic
 process with continuous local time, as in equation
 \eqref{eq:LT_defn}, we define
 \begin{align}
  m^X_s := \inf\left\{t\in [0,1]\ :\ X(t) =
  A(s)\rm{\ and\ }\frac{\ell_t(A(s))}{\ell_1(A(s))} > U\right\}\;
  \rm{for\ }s\in [0,1],
 \end{align}
 where $U$ is an independent \emph{Uniform}$[0,1]$ random variable.
\end{Def}

The analogy between $m_s$ and the quantile permutation is flawed
because $m_s$ requires additional randomization in its definition. But
there can be no bijection from $[0,1]$ to itself which has all of the
properties we would want in a quantile permutation; so we must settle
for $m_s$.

\begin{Thm}
 Let $(X(t),\ t\in [0,1])$ be a L\'evy process, and let $A$ be the
 quantile function of its occupation measure, as in equation
 \eqref{eq:QF_OM_defn}. Fix $T\in [0,1]$ and define
 $$X'(t) := X(t+T) - X(T)\rm{\ for\ }t\in [0,1-T].$$
 Then
 \begin{align}\label{eq:Dassios}
  (X(1),\ A(T)) &\stackrel{d}{=} (X(T) + X'(1-T),\ \sup_{t\in [0,T]}X(t) + \inf_{t\in [0,1-T]}X'(t))
 \end{align}
 (Dassios, 1996\cite{Dassios95,Dassios96}). If $X$ is Brownian bridge
 plus drift, then
 \begin{align}\label{eq:Chaumont}
  \left(\begin{array}{c}
   X(1)\\
   A(T)\\
   m_T
  \end{array}\right) \stackrel{d}{=}
  \left(\begin{array}{c}
   X(T) + X'(1-T)\\
   \sup_{t\in [0,T]}X(t) + \inf_{t\in [0,1-T]}X'(t)\\
   m^{X}_T + m^{X'}_0
  \end{array}\right).
 \end{align}
 (Chaumont, 2000\cite{Chaumont00}).
\end{Thm}

Various path transformation-based proofs of \eqref{eq:Dassios} were
obtained by Embrechts, Rogers, and Yor\cite{EmbrRogeYor95} in the
Brownian case and by Bertoin et.\ al.\cite{BertChauYor97} in the
L\'evy case. Chaumont proved \eqref{eq:Chaumont} with a
continuous-time analogue to the Chaumont transform described
above. These results have applications to finance in the pricing of
Asian options. For a discussion of these applications see
Dassios\cite{Dassios95,Dassios96,Dassios05} and references therein.

Beyond connections in the literature around fluctuations of random
walks and Brownian motion, we also find links between the quantile
transform and discrete versions of Tanaka's formula. Such formulae
have previously been observed by Kudzma\cite{Kudzma82}, Cs\"org\"o and
Rev\'esz\cite{CsorReve85_1}, and Szabados\cite{Szabados90}. See also
\cite{SzabSzek09}. The quantile transformed path may be thought of as
interpolating between points specified by Tanaka's formula. This
connection is made in section \ref{sec:saw}.

\bibliographystyle{plain}
\bibliography{BijectionPf}

\end{document}